\documentclass[12pt,letterpaper]{article}

\usepackage{renstyles}

\graphicspath{{./}{../}{../Figures/}{./Figures}}

\newcommand{\Ai}{{\bf (A-i)}}
\newcommand{\Aii}{{\bf (A-ii)}}
\newcommand{\Aiii}{{\bf (A-iii)}}
\newcommand{\Aiv}{{\bf (A-iv)}}

\title{Improving Learning of Green's Functions Through Kernel Decomposition}
\author{
    Beiji Chen \thanks{Department of Applied Physics and Applied Mathematics, Columbia University, New York, NY 10027; bc3073@columbia.edu}
    \and 
    Qiang Du \thanks{Department of Applied Physics and Applied Mathematics, and Data Science Institute, Columbia University, New York, NY 10027; Email: qd2125@columbia.edu}
    \and 
    Kui Ren \thanks{Department of Applied Physics and Applied Mathematics, Columbia University, New York, NY 10027; kr2002@columbia.edu} 
    \and
    Tian-Yi Zhou \thanks{Data Science Institute, Columbia University, New York, NY 10027; tz2699@columbia.edu}
} 

\begin{document}
\date{}

\maketitle

\begin{abstract}
This work analyzes a framework for learning the solution operator, or equivalently, the Green's function, of a general self-adjoint second-order elliptic operator $\cL$. Subtracting the singular part of the Green's function, either known analytically or supplied by an explicit Laplace parametrix, recasts the learning of $\cL^{-1}$ as learning an $\cL$-harmonic extension operator on the boundary. We formulate the process as operator learning between reproducing kernel Hilbert spaces (RKHSs) to show that the kernel decomposition strategy raises the algebraic input rate by two orders and, under uniform analytic regularity of the induced boundary data together with compatible output spaces, upgrades the output rate from algebraic to exponential. We extend the framework further to semilinear problems via an iterative decomposition in which the only learned object remains a linear operator, ensuring that the improved linear rates translate directly to the nonlinear iteration. Numerical experiments are provided to validate the theoretical development.
\end{abstract}




\section{Introduction}
\label{SEC:Intro}

Let $\Omega\subset\bbR^d$ ($d\ge2$) be a bounded domain with an analytic boundary $\partial\Omega$. Let $u=u(x)$ be the solution to the homogeneous Dirichlet problem
\begin{equation}\label{EQ:Diff}
 \cL u=f,\; \mbox{ in }\, \Omega,\; \mbox{ with }\, \cL =-\nabla\cdot(\gamma\nabla)+\sigma,
\quad \mbox{and} \quad u=0,\; \mbox{on}\, \partial\Omega
\end{equation}
for a given source function $f=f(x)$. The given coefficients $\gamma$ and $\sigma$ are assumed to be real-analytic on $\overline\Omega$, satisfying the following uniform ellipticity and nonnegativity conditions:
\begin{equation}\label{EQ:Coeff Cond}
   0<\gamma_0\le\gamma(x)\le\gamma_1<\infty,
   \qquad \sigma(x)\ge0,
   \qquad x\in\overline\Omega .
\end{equation}
Let $T: H^s(\Omega) \to H^{s+2}(\Omega)\cap H_0^1(\Omega)$, $s\ge 0$, be the solution operator to~\eqref{EQ:Diff}. Then it is well-known that the solution has the Green's function representation
\begin{equation}\label{EQ:Diff Sol}
    u(x)=(Tf)(x)=\int_\Omega G(x,y)\,f(y)\,dy,\quad \forall f\in H^s(\Omega)
\end{equation}
where $G=G(x,y)$ is the Green's function for problem~\eqref{EQ:Diff}.

Learning the solution operator $T$, or equivalently, its kernel $G$, from observed source-solution pairs has become a much-studied problem in scientific machine learning, because the analytic form of $G$ is rarely available for complex geometries or heterogeneous media even when the governing equation  \eqref{EQ:Diff Sol} is fully known~\cite{BoEaTo-SR22, BoKiShTo-JMLR22, GiLuBrKu-SR21}. 

Data-driven approaches to such solution operators fall into two broad families. The first parameterizes the operator by a neural architecture, including DeepONet~\cite{LuJiPaZhKa-NMI21, WaWaPe-SA21}, Fourier neural operators~\cite{LiKoAzLiStAn-ICLR21, KoLaMi-JMLR21}, graph neural operators~\cite{LiKoAzLiBaSuAn-arXiv20}, and their physics-informed variants~\cite{GoAnZaRa-JMPS22, KoLiLiAzBhStAn-JMLR23}, with physics-informed neural networks~\cite{RaPeKa-JCP19,KaKeLuPeWaYa-NRP21,MiMo-IMA22} pursuing the same strategy for a single equation, or learns the Green's function directly through deep or spectral parameterizations~\cite{GiLuBrKu-SR21, BoEaTo-SR22, BoKiShTo-JMLR22, BoHaTo-PNAS23,HaLiXiXuYa-SIAM26,TeZhWaJu-MSML22}. These methods are expressive and mesh-flexible but offer limited \emph{a priori} error control and can be data-intensive. The second family builds the operator from reproducing kernels or Gaussian processes~\cite{ChHoOwSt-JCP21, KaDuPrCaRaAu-JMLR16, Stepaniants-JMLR23, BaDaHoOw-JCP24, LeMeMi-arXiv24, MoYoHoOwBo-CMAME25}, with recent extensions to nonlinear and dynamical settings~\cite{Lazar-arXiv24, Withanachchi-arXiv25, JiHeRaPaLu-NC25, XuLoXuYaZhOw-ICML25} and a companion line establishing \emph{a priori} error and sample-complexity guarantees~\cite{BaChHoOwSt-JCP25, JaOsHsHoOw-arXiv25, BoTo-HNA24, LeMeMi-arXiv24}, including a kernel estimator of the Dirichlet solution operator with convergence rates~\cite{HuOr-arXiv26}. It trades some expressivity for sharper theory and data efficiency. Across most of both families the Green's function is treated as a single unstructured kernel to be fit. A few recent works do separate a known singular part from a learned smooth remainder~\cite{LiChHuChChWaSh-CMS23,SuLiZhJuXu-JCP26,HuWaYaZePe-arXiv26,TeGrBe-arXiv26}, and the same idea has been applied to the Dirichlet-to-Neumann map~\cite{LiYiZh-arXiv26}. These constructions are neural, and only~\cite{LiYiZh-arXiv26}, which concerns the boundary-to-boundary map rather than the Green's function, comes with \emph{a priori} error bounds. The present work supplies the corresponding analysis for the Green's function itself: a computable decomposition for variable coefficients and approximation rates that quantify the gain.

The structure utilized in the present work is embedded in the Green's function itself. For a self-adjoint second-order elliptic operator $\cL$ of the form given in \eqref{EQ:Diff}, the Dirichlet Green's function splits as $G=\Phi^\cL-H^\cL$ into a fundamental solution $\Phi^\cL$, which carries the entire source singularity, and a corrector $H^\cL$ that is $\cL$-harmonic, and therefore real-analytic, inside the domain. The numerically hard-to-approximate part of $G$ is thus exactly its singular part, while the part that must be learned is regular, so subtracting the singular part before learning should make the problem 
more easily solvable. Let $T_{\Phi^\cL}$ and $T_{H^\cL}$ be the integral operators corresponding to the kernels $\Phi^\cL$ and $H^\cL$ respectively (with $T_{\Phi^\cL}$ defined by restricting $\Phi^\cL$ to $\Omega\times\Omega$). Then, the solution operator is given by
\begin{equation}\label{EQ:T Split}
T=T_{\Phi^\cL}-T_{H^\cL}.
\end{equation} 
When $\Phi^\cL$ is known (for instance, when $\cL$ is the Dirichlet Laplacian), the problem of learning $T$ becomes an {\em easier} task of learning the operator $T_{H^\cL}$. The main task for us is to convert such an intuition into a provable gain, namely, a faster rate of approximation. This motivates a more general operator-learning setting where the error is governed quantitatively by the regularity of the object being learned.
We develop an abstract framework for such a setting in the context of reproducing kernel Hilbert spaces (RKHSs). 
 
Let $(\cH_X,\cX,K_X)$ and $(\cH_Y,\cY,K_Y)$ be RKHSs on domains $\cX$ and $\cY$ with kernels $K_X$ and $K_Y$ and 
the operator $ T : \cH_X \rightarrow \cH_Y$. Our goal is to construct a bounded surrogate operator 
\begin{equation}
    \wh T : \cH_X \rightarrow \cH_Y
\end{equation}
that approximates the true underlying operator $T$ based on a set of $N$ paired data points:
\begin{equation}\label{EQ:Data}
    \{(f_\ell, u_\ell:=Tf_\ell)\}_{\ell=1}^N.
\end{equation}

The main contribution of this work is twofold: to establish, for $\sigma\not\equiv0$, the $a$-orthogonal direct-sum structure $\cH_N\cong\cH_G\oplus_a\cH_{H^\cL_a}$ of the energy space (Theorem~\ref{THM:Direct Sum Gen}), and to exploit the pointwise splitting $G=\Phi^\cL-H^\cL$ to reformulate the learning problem. The latter is used to construct
\begin{equation}
    \wh T_{H^\cL} : \wt\cH_{\wt X} \rightarrow \wt\cH_{\wt Y}
\end{equation}
that approximates the corrector operator $T_{\Phi^\cL}-T=T_{H^\cL}$, between reformulated input/output RKHSs $\wt\cH_{\wt X},\wt\cH_{\wt Y}$, to be identified in Sections~\ref{SEC:Elliptic}--\ref{SEC:Computable-Decom} for the problem \eqref{EQ:Diff Sol}, based on a set of $N$ paired data points:
\begin{equation}\label{EQ:Data Residual}
    \{(f_\ell, w_\ell:=T_{\Phi^\cL} f_\ell-u_\ell)\}_{\ell=1}^N.
\end{equation}
We will present \emph{a priori} \emph{approximation} bounds to control the error of the best surrogate representable, independent of the sample size $N$. The data enter only through the fitting of the resulting finite-dimensional map, a step that the framework treats as a black box.

The remainder of the paper is organized as follows. We first recall the RKHS background in Section~\ref{SEC:Prelim} used throughout the work. We then develop an abstract framework of learning operators between RKHSs in Section~\ref{SEC:Learning-RKHS}. In Section~\ref{SEC:Elliptic}, we specialize the framework to the problem of learning $T$ in~\eqref{EQ:Diff Sol} based on kernel decomposition and show that such a decomposition leads to sharper rates. A computable version of the decomposition is proposed in Section~\ref{SEC:Computable-Decom}. Section~\ref{SEC:Nonlinear} extends the construction to nonlinear (semilinear) solution operators, where the fitted map is no longer exactly representable by a matrix, and where an iterative use of the learned linear operator preserves the applicable linear rates under uniform assumptions. Finally, we present some numerical simulations in Section~\ref{SEC:Numer} to illustrate the approximation gains. Concluding remarks are offered in Section~\ref{SEC:Concl}.

\section{Preliminaries on RKHS}
\label{SEC:Prelim}


We recall here some basic results from the theory of reproducing kernel Hilbert spaces (RKHS) for the sake of self-containedness. We refer interested readers to~\cite{Aronszajn-TAMS50, BeTh-Book11,SaSa-Book16} and references therein for more technical details. We only work with spaces of real-valued functions. Throughout this section, $\Omega$ denotes an arbitrary nonempty set. 

\begin{definition}\label{DEF:PD Kernel}
A symmetric function $K:\Omega\times \Omega\to\bbR$ on a nonempty set $\Omega$ is a \emph{positive-definite kernel} if for every finite collection $\{x_i\}_{i=1}^m\subset \Omega$ and coefficients $\{c_i\}_{i=1}^m\subset\bbR$,
\[
\sum_{i,j=1}^m c_i c_j K(x_i,x_j) \;\ge\; 0.
\]
\end{definition}

\begin{definition}\label{DEF:RKHS}
A Hilbert space $\cH$ of functions $f:\Omega\to\bbR$ is a \emph{reproducing kernel Hilbert space} if there exists a kernel $K:\Omega\times \Omega\to\bbR$, called the \emph{reproducing kernel}, such that $K(x,\cdot)\in\cH$ for every $x\in \Omega$ and
\begin{equation}\label{EQ:Reproducing}
f(x) = \langle f,\, K(x,\cdot)\rangle_{\cH} \quad \text{for all } f\in\cH,\; x\in \Omega.
\end{equation}
\end{definition}

The reproducing property~\eqref{EQ:Reproducing} immediately implies that point evaluation is a bounded linear functional in the sense that  $|f(x)|\le\|K(x,\cdot)\|_{\cH}\|f\|_{\cH}$. If, in addition, $K$ is continuous and bounded on the diagonal (as holds for every kernel used in this paper), it follows that $\cH$ embeds continuously into $\cC^0(\Omega)$. It is also well-known that not only is the reproducing kernel necessarily symmetric and positive definite, but also the converse is true, as given in the following Moore-Aronszajn Theorem
\cite{Aronszajn-TAMS50,BeTh-Book11}.
\begin{theorem}[Moore-Aronszajn]\label{THM:Moore Aronszajn}
For every positive-definite kernel $K:\Omega\times \Omega\to\bbR$, there exists a unique RKHS $\cH_K$ with reproducing kernel $K$, given by $\cH_K = \overline{\mathrm{span}\{K(x,\cdot): x\in \Omega\}}$.
\end{theorem}

The following Aronszajn sum theorem characterizes the RKHS generated by a sum of kernels.
\begin{theorem}[Aronszajn Sum]\label{THM:Aronszajn Sum}
If $K_1$ and $K_2$ are reproducing kernels of $\cH_1$ and $\cH_2$, then $K=K_1+K_2$ is the reproducing kernel of the sum $\cH_1+\cH_2=\{f_1+f_2:f_1\in\cH_1,\ f_2\in\cH_2\}$, with norm
\[
\|f\|_{\cH}^2 := \min_{\substack{f=f_1+f_2\\ f_1\in\cH_1,\,f_2\in\cH_2}}\bigl(\|f_1\|_{\cH_1}^2+\|f_2\|_{\cH_2}^2\bigr).
\]
\end{theorem}
When $\cH_1\cap\cH_2=\{0\}$, the sum is an orthogonal direct sum. This is the case established for the pairs used in the rest of the paper.

A central example throughout this paper is the Sobolev space $H^s(\Omega)$ with $s>d/2$ of functions on $\Omega \subseteq \bbR^d$: by the Sobolev embedding theorem, it embeds continuously into $\cC^0(\Omega)$, and is therefore an RKHS~\cite{Wendland-Book04}.
 
A recurring task in what follows is to approximate a function in $\cH$ from its values on a finite set, and the accuracy of any such approximation is governed by a single geometric quantity, the \emph{power function}. Fix a finite set $W=\{w_1,\dots,w_m\}\subset \Omega$. By the reproducing property~\eqref{EQ:Reproducing}, the kernel section $K(\cdot,x)$ is the Riesz representer of the point-evaluation functional $f\mapsto f(x)$. Let $S_W$ be the (finite-dimensional) kernel subspace associated with $W$, that is
\[
   S_W := \operatorname{span}\{K(\cdot,w):w\in W\}\subset\cH\,.
\]
We write $P_W:\cH\to S_W$ for the $\cH$-orthogonal projection onto it. Then, for a given $f$, $P_W f$ is the minimum-norm element of $\cH$ that interpolates $f$ at the points of $W$ (that is, $S_W$ is the canonical finite-dimensional surrogate for a function known only through its samples on $W$). How faithfully this surrogate reproduces point evaluation at an arbitrary location $x\in\Omega$ is measured by the residual of $K(\cdot,x)$ after projection onto $S_W$, and this residual is the power function \cite{Wendland-Book04}.

\begin{definition}
Let $\cH$ be a RKHS on $\Omega$ with kernel $K$, let $W\subset \Omega$ be finite, and let $S_W$ and $P_W$ be as above. The \emph{power function} of the pair $(K,W)$ is
\[
   P_{K,W}(x) := \bigl\|K(\cdot,x) - P_W K(\cdot,x)\bigr\|_{\cH}, \qquad x\in \Omega\,.
\]
\end{definition}
Therefore, $P_{K,W}(x)$ vanishes on $W$ and grows as $x$ moves away from $W$. It depends only on the kernel $K$ and the geometry of $W$. The relevance of the power function is that, through the reproducing property, it controls the pointwise approximation error from $S_W$ uniformly over $\cH$. This is the standard power-function error bound; see, for instance,~\cite{GolombWeinberger1959,Wendland-Book04}.

\begin{theorem}\label{THM:Power Func}
Let $\cH$ be an RKHS on $\Omega$ with kernel $K$, $W\subset\Omega$ a finite set, and
$P_W$ the orthogonal projection onto $\operatorname{span}\{K(\cdot,w):w\in W\}$.
Then for any $g\in\cH$,
\begin{equation*}
\|g - P_{W} g\|_{L^\infty(\Omega)} \leq \sup_{x\in\Omega}P_{K,W}(x)\cdot\|g\|_{\cH},
\end{equation*}
and the constant $\sup_{x\in\Omega}P_{K,W}(x)$ is optimal.
\end{theorem}
Therefore, the supremum of the power function, $\sup_{x\in\Omega}P_{K,W}(x)$, quantifies the worst-case approximability of $\cH$ from $S_W$. It is clear that the supremum of $P_{K,W}$ is controlled by, besides the smoothness of the kernel $K$, how dense $W$ is in $\Omega$, which is measured by the \emph{fill distance} of $W$ in $\Omega$, defined as
\begin{equation*}
    h_W := \sup_{x\in \Omega}\inf_{w\in W}\|x-w\|\,.
\end{equation*} 
The following theorem, adapted from~\cite{Wendland-Book04}, makes this dependence quantitative in the two regimes relevant to this paper.
\begin{theorem}[{\cite[Theorems~11.13 and~11.22]{Wendland-Book04}}]\label{THM:Power Func Decay}
Let $\Omega\subset\bbR^d$ be a bounded domain whose boundary is Lipschitz and which satisfies an interior cone condition, and let $K$ be a symmetric positive-definite kernel on $\Omega$ with native space $\cH_K$. Then there exist constants $h_0>0$ and $C>0$, depending only on $\Omega$, $d$, and $K$ but not on the set $W$, such that the following hold for every finite $W\subset\Omega$ with fill distance $h_W\le h_0$.\\[1ex]
(i) If $\cH_K$ embeds continuously into $H^{s}(\Omega)$ for some $s>d/2$, the Sobolev sampling inequality applied to the projection residual gives
\begin{equation}\label{EQ:Power Func Bound Alge}
   \sup_{x\in\Omega} P_{K,W}(x) \;\le\; C\, h_W^{\,s-d/2}.
\end{equation}
(ii) If $\Omega$ is a cube and $K(x,y)=\exp(-\epsilon^2|x-y|^2)$ is a Gaussian kernel with fixed $\epsilon>0$, then for every compact subset $\Omega'\Subset\Omega$ there exist constants $C,c>0$ (depending also on $\Omega'$) such that
\begin{equation}\label{EQ:Power Func Bound Exp}
   \sup_{x\in\Omega'} P_{K,W}(x) \;\le\; C\, e^{-c/h_W}.
\end{equation}
\end{theorem} 

\section{Learning linear operators between RKHSs}
\label{SEC:Learning-RKHS}

We now formalize the process of learning operators between RKHSs. This formulation is standard and has been developed for more general settings (see, for instance, ~\cite{BaDaHoOw-JCP24,LiDu-NMPDE26} and references therein), and will be used in later sections to highlight the main point of this work: one can reformulate the learning problem in ``good" RKHSs to make the learning more efficient. 

\subsection{The setup}

We fix discrete sets of sampling points $U=\{u_1,\ldots,u_{m_X}\}\subset\cX$ and $V=\{v_1,\ldots,v_{m_Y}\}\subset\cY$. Associated with $U$, we define the \emph{Gram matrix} $K_{UU}\in\bbR^{m_X\times m_X}$ by $(K_{UU})_{ij}=K_X(u_i,u_j)$, and the  \emph{evaluation vector} $\mathbf{f}_U=[f(u_1),\dots,f(u_{m_X})]^\top$ for any $f\in\cH_X$. 

Throughout the rest of the work, we assume the following about the operator $T$, its input space, the kernels, and the sampling points. 
\begin{itemize}[leftmargin=3.5em, itemsep=2pt, labelsep=0.4em]

\item[\hypertarget{Assum-A1}{\Ai\,\,}] 
Let $\cM\subset\cH_X$ denote the collection of source terms $f$. $\cM$ is bounded in $\cH_X$ in the sense that there exists $C_f>0$ such that $\|f\|_{\cH_X}\le C_f$ for all $f\in\cM$.

\item[\hypertarget{Assum-A2}{\Aii\,\,}] The operator $T:\cH_X\to\cH_Y$ satisfies the following conditions on $\cM$:
\begin{enumerate}[label=(\roman*)]
    \item \textbf{Boundedness:} $M_Y:=\sup_{f\in\cM}\|T(f)\|_{\cH_Y}<\infty$.
    \item \textbf{$L^\infty$-Lipschitz continuity:} There exists a constant $L_T>0$ such that
    \begin{equation*}
        \|T(f_1)-T(f_2)\|_{L^\infty(\cY)} \le L_T\|f_1-f_2\|_{L^\infty(\cX)},
        \quad \forall\, f_1,f_2\in \cB_{C_f},
    \end{equation*}
    with $\cB_{C_f}:=\{f\in\cH_X:\|f\|_{\cH_X}\le C_f\}$. 
\end{enumerate}

\item[\hypertarget{Assum-A3}{\Aiii\,\,}]  The kernels $K_X$ and $K_Y$ are strictly positive-definite, and the sampling points in $U$ and $V$ are distinct.

\item[\hypertarget{Assum-A4}{\Aiv\,\,}] 
Let $h_X$ and $h_Y$ be the fill-distances of $U\subset\cX$ and $V\subset\cY$, respectively. The power functions associated with $(K_X, U)$ and $(K_Y, V)$ satisfy
\begin{equation*}
\sup_{x\in\cX} P_{K_X,U}(x)\le\eps_X(h_X), 
\qquad 
\sup_{y\in\cY} P_{K_Y,V}(y)\le\eps_Y(h_Y),
\end{equation*}
with functions $\eps_X,\eps_Y:[0,\infty)\to[0,\infty)$ nondecreasing and satisfying $\eps_X(h),\eps_Y(h)\to 0$ as $h\to 0$.

\end{itemize}
Under the assumption~\hyperlink{Assum-A3}{\Aiii}, the Gram matrices $K_{UU}$ and $K_{VV}$ are strictly positive-definite. This 
positive-definiteness assumption implies the invertibility, which is an idealization adopted for clarity of exposition. In practice, the Gram matrix becomes increasingly ill-conditioned as the number of sampling points grows, and the encoders are stabilized by Tikhonov regularization.

With this setup, we define the kernel subspace  $S_X=\operatorname{span}\{K_X(\cdot,u_i)\}_{i=1}^{m_X}\subset\cH_X$ and the orthogonal projection onto $S_X$ as
\begin{equation*}
P_Xf=\sum_{i=1}^{m_X}(\bc_X)_i\, K_X(\cdot,u_i), \qquad \bc_X=K_{UU}^{-1}\mathbf{f}_U,
\end{equation*}
where the coefficient formula $\bc_X = K_{UU}^{-1}\mathbf{f}_U$ follows from the reproducing property and the normal equations for orthogonal projection onto $S_X$. Analogously, $V$ defines the Gram matrix $K_{VV}\in\bbR^{m_Y\times m_Y}$ with $(K_{VV})_{ij}=K_Y(v_i,v_j)$, the evaluation vector $\mathbf{g}_V=[g(v_1),\dots,g(v_{m_Y})]^\top$, and the subspace $S_Y=\operatorname{span}\{K_Y(\cdot,v_j)\}_{j=1}^{m_Y}\subset\cH_Y$ with orthogonal projection $P_Yg=\sum_j(\mathbf{d}_Y)_j K_Y(\cdot,v_j)$, $\mathbf{d}_Y=K_{VV}^{-1}\mathbf{g}_V$.

We represent these coefficient maps directly as operators. Specifically, the \emph{encoding operator} $\cP_X:\cH_X\to\bbR^{m_X}$ is defined as $\cP_X(f) :=\bc_X$. The \emph{decoding operator} $\cR_Y:\bbR^{m_Y}\to S_Y\subset\cH_Y$ mapping the coefficient vector back to a function in $\cH_Y$ is defined by
\begin{equation}
   \cR_Y(\mathbf{d}_Y):= \sum_{j=1}^{m_Y}(\mathbf{d}_Y)_j\, K_Y(\cdot,v_j), \qquad \mathbf{d}_Y\in \bbR^{m_Y}.
\end{equation}

Given a dataset $\{(f_\ell, u_\ell)\}_{\ell=1}^N$, we encode each pair to the coefficient vectors $\bc_X^\ell=\cP_X(f_\ell)$ and $\mathbf{d}_Y^\ell=K_{VV}^{-1} \mathbf{u}_{V}^\ell$, and fit the coordinate map $\bc_X\mapsto\mathbf{d}_Y$. Every operator learned in this paper is linear, so the coordinate map is itself linear and is fitted by a matrix  $\bM\in\bbR^{m_Y\times m_X}$.  The full approximate operator is then given by
\begin{equation}\label{EQ:Approx Opera}
    \wh T = \cR_Y \circ \bM \circ \cP_X.
\end{equation}
The linearity means that the error analysis has only two contributions: the projection of the output onto $S_Y$, $\cE_{\rm out}$, and the projection of the input onto $S_X$, $\cE_{\rm in}$. Indeed, define the
finite-dimensional target matrix (using the input-side decoding operator
$\cR_X(\bc):=\sum_{i=1}^{m_X}c_i\,K_X(\cdot,u_i)$)
\begin{equation}\label{EQ:GStar}
   \bM^\star:=K_{VV}^{-1}\bA,\qquad
   A_{ki}=T(K_X(\cdot,u_i))(v_k),
\end{equation}
so that, by linearity of $T$ and $\cR_X$,
$\cR_Y\bM^\star\cP_X=P_YTP_X$ exactly: the matrix class represents the projected operator with no residual. With $\bM=\bM^\star$, the
triangle inequality gives, for any $f\in\cM$,
\begin{equation}\label{EQ:Error Decomp}
\|T(f) - \wh T(f)\|_{L^\infty(\cY)}
\le \underbrace{\|T(f) - P_Y T(f)\|_{L^\infty(\cY)}}_{\cE_{\mathrm{out}}(f)} + \underbrace{\|P_Y T(f) - P_Y T P_X (f)\|_{L^\infty(\cY)}}_{\cE_{\mathrm{in}}(f)}.
\end{equation}
The two terms are bounded in Proposition~\ref{PROP:Output Error} and
Proposition~\ref{PROP:Input Error}. Whether the data determine $\bM^\star$ is the identifiability question settled in
Section~\ref{SUBSEC:Linear Identifiability}.

We measure the approximation errors in the $L^\infty$ norm, a choice that is natural in the RKHS setting for two reasons. First, the power function controlling the projection errors below is defined through pointwise kernel approximation and is therefore intrinsically uniform. The reproducing property further ensures that RKHS bounds transfer to $L^\infty$ bounds without loss. Second, different operators and approximation schemes typically live in different RKHSs, and $L^\infty$ provides a method-agnostic norm in which their guarantees remain comparable.

\subsection{The general error bounds}
 
Applying Theorem~\ref{THM:Power Func} to the output and input spaces directly yields the following two bounds.
\begin{proposition}\label{PROP:Output Error}
Under assumptions~\hyperlink{Assum-A1}{\Ai}--\hyperlink{Assum-A4}{\Aiv}, for any $f\in\cM$,
\begin{equation}\label{EQ:Output Error}
\cE_{\mathrm{out}}(f) \le M_Y\,\eps_Y(h_Y)\,.
\end{equation}
\end{proposition}
\begin{proof}
Since $f \in \cM$, assumption~\hyperlink{Assum-A2}{\Aii} gives $T(f) \in \cH_Y$ with
$\|T(f)\|_{\cH_Y} \le M_Y$. We then apply Theorem~\ref{THM:Power Func} to $T(f)$ with the sampling set $V$ to get
\[
\|T(f) - P_Y T(f)\|_{L^\infty(\cY)}
\le \Bigl(\sup_{y\in\cY} P_{K_Y,V}(y)\Bigr)\|T(f)\|_{\cH_Y}
\le \eps_Y(h_Y)\, M_Y\,
\]
which is the desired result.
\end{proof}

\begin{proposition}\label{PROP:Input Error}
Under assumptions~\hyperlink{Assum-A1}{\Ai}--\hyperlink{Assum-A4}{\Aiv}, for any $f\in\cM$, we have
\begin{equation}\label{EQ:Input Error}
\cE_{\mathrm{in}}(f)
\le C_P L_T C_f\,\eps_X(h_X),
\end{equation}
where $C_P := \|P_Y\|_{L^\infty(\cY)\to L^\infty(\cY)}$ is the 
$L^\infty$-stability constant of $P_Y$, while $L_T$ is the constant from assumption~\hyperlink{Assum-A2}{\Aii} and $C_f$ that from assumption~\hyperlink{Assum-A1}{\Ai}.
\end{proposition}
\begin{proof}
By the $L^\infty$ stability of $P_Y$ and the Lipschitz continuity of $T$
(in assumption~\hyperlink{Assum-A2}{\Aii}), we have
\[
\|P_Y T(f) - P_Y T(P_X f)\|_{L^\infty(\cY)}
\le C_P\, \|T(f) - T(P_X f)\|_{L^\infty(\cY)}
\le C_P L_T\, \|f - P_X f\|_{L^\infty(\cX)}.
\]
We then apply Theorem~\ref{THM:Power Func} to $f$ with sampling set $U$ and
using $\|f\|_{\cH_X} \le C_f$ (assumption~\hyperlink{Assum-A1}{\Ai}) to have
\[
\|f - P_X f\|_{L^\infty(\cX)} \le \eps_X(h_X)\,\|f\|_{\cH_X} \le \eps_X(h_X)\, C_f.
\]
Combining the two inequalities gives the result.
\end{proof}

\begin{remark}
The constant $C_P=\|P_Y\|_{L^\infty(\cY)\to L^\infty(\cY)}$ entering Proposition~\ref{PROP:Input Error} is the Lebesgue constant of kernel interpolation on $V$. The orthogonal projection $P_Y$ coincides with the kernel interpolant at $V$: the residual $g-P_Yg$ is orthogonal to each $K_Y(\cdot,v_j)$, so by the reproducing property $g(v_j)=(P_Yg)(v_j)$ for all $j$. Writing $P_Yg=\sum_{j=1}^{m_Y}g(v_j)\,\chi_j$ in the cardinal basis $\chi_j\in S_Y$ with $\chi_j(v_k)=\delta_{jk}$, we obtain
\begin{equation*}
   C_P=\sup_{y\in\cY}\sum_{j=1}^{m_Y}|\chi_j(y)| .
\end{equation*}
Thus $C_P$ depends on $K_Y$ and $V$, not on the function being approximated, and is finite for fixed nodes when the cardinal functions are bounded. When stating rates with mesh-independent prefactors, we assume that the chosen interpolation family has $C_P\le C_{\mathrm{stab}}$ uniformly. Otherwise the factor $C_P$ must be retained, and convergence of the input term requires $C_P(h_Y)\eps_X(h_X)\to0$. Corollary~\ref{COR:Global Rates} below avoids this stability factor for the harmonic extension by using its boundedness between the chosen RKHSs.
\end{remark}

Combining the two propositions, we see that, under assumptions~\hyperlink{Assum-A1}{\Ai}-\hyperlink{Assum-A4}{\Aiv}, for any $f\in\cM$, the surrogate~\eqref{EQ:Approx Opera} with the exact matrix
$\bM=\bM^\star$ of~\eqref{EQ:GStar} obeys
\begin{equation}\label{EQ:Error Summary}
\|T(f)-\wh T(f)\|_{L^\infty(\cY)}
\;\le\;
M_Y\,\eps_Y(h_Y)
\;+\;
C_P L_T C_f\,\eps_X(h_X)\,.
\end{equation}

Let us emphasize again that the bound~\eqref{EQ:Error Summary} is an \emph{a priori approximation} bound: it controls the error of the best surrogate $\wh T=\cR_Y\circ \bM \circ\cP_X$ representable by the matrix class, and it is consequently free of the sample size $N$ of~\eqref{EQ:Data}. Indeed, the sampling sets $U,V$ are design choices made independently of the data, and the matrix $\bM^\star$ exists independently of the data.  The data enter only in fitting $\bM$, which the framework deliberately treats as a black box. If we write $\wh T_N$ for the surrogate actually fitted from $N$ pairs, the full error splits as
\[
   \|T-\wh T_N\|
   \;\le\;\underbrace{\|T-\wh T\|}_{\text{approximation}}
   \;+\;\underbrace{\|\wh T-\wh T_N\|}_{\text{estimation}},
\]
together with optimization and, for noisy data, variance contributions. Only the first term is analyzed here. Controlling the second is a statistical question, for which sample-complexity results of the type in~\cite{BaChHoOwSt-JCP25,JaOsHsHoOw-arXiv25,BoKiShTo-JMLR22} would be required and is left to future work.

\subsection{Identifiability from data}
\label{SUBSEC:Linear Identifiability}

The bound~\eqref{EQ:Error Summary} concerns the approximation error of
$P_YTP_X$, which is represented exactly by the matrix
$\bM^\star=K_{VV}^{-1}\bA$ defined in~\eqref{EQ:GStar}. Whether
$\bM^\star$ can be \emph{recovered from data} is a separate question, and it is the one point at which the sample size $N$ of~\eqref{EQ:Data} enters the
theory explicitly.
\begin{proposition}\label{PROP:Linear Identifiability}
Let $\bM^\star:=K_{VV}^{-1}\bA$ be the target matrix of~\eqref{EQ:GStar}, and let the pairs $\{(f_\ell, u_\ell)\}_{\ell=1}^N$ be noiseless, with sources
$f_\ell\in S_X=\operatorname{span}\{K_X(\cdot,u_i)\}_{i=1}^{m_X}$. Then $\bM^\star$ is determined by the data
if and only if
\begin{equation}\label{EQ:Design Rank}
   \operatorname{rank}\bigl[\,\cP_Xf_1,\ \dots,\ \cP_Xf_N\,\bigr]=m_X,
\end{equation}
which in particular requires $N\ge m_X$. Moreover, the minimal design is explicit: taking $N=m_X$ with $f_\ell=K_X(\cdot,u_\ell)$ gives $\cP_Xf_\ell=e_\ell$, so that $\mathbf{d}_Y^\ell$ is precisely the $\ell$-th column of $\bM^\star$ and the operator is recovered in a single pass.
\end{proposition}
\begin{proof}
For $f_\ell\in S_X$ one has $P_Xf_\ell=f_\ell$, so encoding the output gives exactly the linear constraint $\mathbf{d}_Y^\ell=\bM\,\cP_Xf_\ell$. The rows of $\bM$ are determined exactly when the vectors $\{\cP_Xf_\ell\}_{\ell=1}^N$ span $\bbR^{m_X}$, which is~\eqref{EQ:Design Rank} and forces $N\ge m_X$. For the minimal design, $f_\ell=K_X(\cdot,u_\ell)$ gives $(\mathbf{f}_{\ell,U})_j=K_X(u_j,u_\ell)=(K_{UU})_{j\ell}$, i.e.\ $\mathbf{f}_{\ell,U}=K_{UU}e_\ell$,
whence $\cP_Xf_\ell=K_{UU}^{-1}K_{UU}e_\ell=e_\ell$ and $\mathbf{d}_Y^\ell=\bM^\star e_\ell$.
\end{proof} 

Condition~\eqref{EQ:Design Rank} ensures that the data uniquely determine $\bM^\star$. The learned operator reproduces $P_YTP_Xf$, with only output error
for $f\in S_X$ and an additional input error otherwise. This is consistent with the numerical observation in Section~\ref{SEC:Numer} that the error curves level off on a plateau set by the projection and discretization errors rather than decreasing to zero. Once it holds, further samples do not improve the bound~\eqref{EQ:Error Summary}, which is why $N$ is absent from it. The probing design $f_i=K_X(\cdot,u_i)$ also has the best possible conditioning, since the encoded inputs form the standard basis.

\section{Kernel decomposition for elliptic PDE}
\label{SEC:Elliptic} 

We now specialize the abstract RKHS framework of
Section~\ref{SEC:Learning-RKHS} to a concrete PDE family, where its rates become fully explicit, and its structure can be exploited. We work directly with
the general self-adjoint, uniformly elliptic, divergence-form operator in~\eqref{EQ:Diff} under the homogeneous Dirichlet boundary condition. The Green's function of $\cL$ decomposes analytically into a universal singular part and a domain-dependent corrector.

\subsection{The energy inner product}
\label{SUBSEC:PDE Setup}

The Dirichlet Green's function $G(\cdot,y)$ solves
\begin{equation}\label{EQ:Green Func Def}
   \cL_x G(x,y)=\delta(x-y)\ \ \text{in }\Omega,
   \qquad G(x,y)=0\ \ \text{on }\partial\Omega .
\end{equation}
Because $\cL$ is self-adjoint, $G$ is symmetric, $G(x,y)=G(y,x)$, and positive by
the maximum principle. It is the kernel associated with the operator $\cL^{-1}$. Moreover, it is positive-definite 
 in the sense that: $$\int_\Omega\int_\Omega f(x)\,G(x,y)\,f(y)\,dxdy=\langle\cL^{-1}f,f\rangle_{L^2(\Omega)}\ge0, 
 \quad \forall f\in L^2(\Omega).
 $$
Note that, for $d\ge2$, the kernel $G$ is not positive-definite in the pointwise sense of the Definition~\ref{DEF:PD Kernel}, since $G(x,x)=+\infty$. 
  
  Standard elliptic regularity~\cite{GiTr-Book00} gives $Tf\in H^{s+2}(\Omega)\cap H_0^1(\Omega)$ with
$\|Tf\|_{H^{s+2}}\lesssim\|f\|_{H^s}$. To guarantee bounded pointwise evaluation we require $s>d/2$ (Sobolev embedding into $\cC^0$, so that $H^s(\Omega)$ and
$H^{s+2}(\Omega)$ are RKHSs~\cite{BeTh-Book11,Wendland-Book04}). We set
\begin{equation*}
   \cH_{\mathrm{src}}:=H^s(\Omega),\qquad
   \cH_{\mathrm{sol}}:=H^{s+2}(\Omega)\cap H_0^1(\Omega),\qquad s>\tfrac{d}{2},
\end{equation*}
so that $T:\cH_{\mathrm{src}}\to\cH_{\mathrm{sol}}$ is a bounded linear operator between RKHSs.

Let us introduce the standard energy inner product, $a(u,v)$ (and the corresponding norm $\|\cdot\|_a$), of $\cL$ as
\begin{equation}\label{EQ:Energy IP}
   a(u,v):=\int_\Omega \bigl(\gamma\,\nabla u\cdot\nabla v + \sigma\,u\,v\bigr)\,dx ,
   \quad \|u\|_a^2:=a(u,u), \quad
   \forall\, u,v\in H^1(\Omega).
\end{equation}
By uniform ellipticity~\eqref{EQ:Coeff Cond}, $\sigma\ge0$, and the Poincar\'e inequality, $a(\cdot,\cdot)$ is an inner product on $H_0^1(\Omega)$ equivalent to the standard one. Throughout, $\Phi$ denotes
the \emph{Laplace} fundamental solution satisfying $-\Delta\Phi=\delta$ in $\bbR^d$:
\begin{equation}\label{EQ:Laplace Fundamental Sol}
   \Phi(x,y)=\begin{cases}
   -\dfrac{1}{2\pi}\ln|x-y|, & d=2,\\[1.5ex]
   \dfrac{1}{d(d-2)\alpha(d)}\dfrac{1}{|x-y|^{d-2}}, & d\ge3,
   \end{cases}
\end{equation}
with $\alpha(d)$ the volume of the unit ball in $\bbR^d$. For notation simplicity, we write $\Phi(x-y)$ interchangeably with $\Phi(x,y)$.

\begin{lemma}\label{LEM:Poisson Assumptions}
Let $\cM=\{f\in\cH_{\mathrm{src}}:\|f\|_{H^s(\Omega)}\le R\}$.
Then $T=\cL^{-1}$ satisfies assumptions~\hyperlink{Assum-A1}{\Ai} and~\hyperlink{Assum-A2}{\Aii} with
\[
M_Y=M_{\mathrm{sol}}:=\sup_{f\in\cM}\|Tf\|_{H^{s+2}(\Omega)}\le C_{\mathrm{ell}}\,R,
   \qquad 0<
   L_T=\|\tau\|_{L^\infty(\Omega)}<\infty,
\]
where $\tau$ is the torsion function solving 
\begin{equation}\label{EQ:Torsion}
\cL\tau=1\; \text{ in }\; \Omega\quad\text{ and }\quad \tau|_{\partial\Omega}=0.
\end{equation}
\end{lemma}
\begin{proof}
The boundedness is a standard elliptic estimate $\|Tf\|_{H^{s+2}}\le C_{\mathrm{ell}}\|f\|_{H^s}$,
leading to $M_{\mathrm{sol}}\le C_{\mathrm{ell}}R$. For the $L^\infty$-Lipschitz
bound, $T$ is linear, so for $f_1,f_2\in\cM$ and $g:=f_1-f_2$,
\begin{equation}
   \|Tf_1-Tf_2\|_{L^\infty(\Omega)}
   =\Bigl\|\int_\Omega G(\cdot,y)\,g(y)\,dy\Bigr\|_{L^\infty(\Omega)}
   \le \|g\|_{L^\infty(\Omega)}\,\sup_{x\in\Omega}\int_\Omega G(x,y)\,dy,
\end{equation}
where $G\ge0$ by the maximum principle (valid since $\sigma\ge0$). 

The function
$\tau(x):=\int_\Omega G(x,y)\,dy$ is exactly the solution of \eqref{EQ:Torsion}, so
\[\sup_{x\in\Omega}\int_\Omega G(x,y)\,dy=\|\tau\|_{L^\infty(\Omega)}=:L_T,\]
which is finite by the maximum principle. Since $\cM$ is a bounded ball, the assumption~\hyperlink{Assum-A1}{\Ai} holds by definition.
\end{proof}

Applying the algebraic power-function estimate~\eqref{EQ:Power Func Bound Alge} to the input and output RKHSs gives
\[
   \eps_X(h_X)\le C_X h_X^{\,s-d/2},
   \qquad
   \eps_Y(h_Y)\le C_Y h_Y^{\,s+2-d/2}.
\]
Substituting these estimates into~\eqref{EQ:Error Summary} yields, for every $f\in\cM$ (with $C_f=R$ for the ball $\cM$ of Lemma~\ref{LEM:Poisson Assumptions}), the bound
\begin{equation}\label{EQ:Total Error Bound}
   \|T(f)-\wh T(f)\|_{L^\infty(\Omega)}
   \le M_{\mathrm{sol}}\,C_Y h_Y^{\,s+2-d/2} + C_P L_T C_f\,C_X h_X^{\,s-d/2}.
\end{equation}
This is the baseline rate we have for the direct learning of $T$ without kernel decomposition. 

\subsection{The energy RKHS structure} 

We now analyze the analytical structure of the Green's function and its induced RKHS. The development mirrors the Dirichlet-inner-product theory of Green's functions as reproducing kernels~\cite{ArSm-AJM57,FaYe-NM11}, with the standard inner product replaced by the energy inner product~\eqref{EQ:Energy IP}. The same energy-space viewpoint underlies operator-adapted wavelets and the kernel-learning recovery of elliptic Green's functions in~\cite{OwSc-Book19,OwYo-arXiv18}. The variable coefficients $\gamma,\sigma$ ride along inside the energy form.

\begin{theorem}\label{THM:HG Gen}
The space $H_0^1(\Omega)$ equipped with the energy inner product~\eqref{EQ:Energy IP} admits $G$ as a \emph{generalized} reproducing kernel in the sense that, for every $u\in C_c^\infty(\Omega)$ and $x\in\Omega$,
\begin{equation}\label{EQ:Repro Energy}
u(x)=a\bigl(u,\,G(x,\cdot)\bigr):=\int_\Omega(\gamma\,\nabla u\cdot\nabla_zG(x,z)+\sigma\,u\,G(x,z))\,dz
\end{equation}
where the right-hand side is understood as the absolutely convergent
integral.
\end{theorem}
\begin{proof}
For $u\in C_c^\infty(\Omega)$ and fixed $x$, integrating by parts and using that $u$ vanishes on $\partial\Omega$,
\[
   a\bigl(u,G(x,\cdot)\bigr)
   =\int_\Omega\bigl(\gamma\,\nabla u\cdot\nabla_z G(x,z)+\sigma\,u\,G(x,z)\bigr)dz
   =\int_\Omega u(z)\,\cL_z G(x,z)\,dz .
\]
The first integral converges absolutely since $|\nabla_zG(x,z)|\lesssim
|x-z|^{1-d}\in L^1_{\mathrm{loc}}$. By the symmetry of $G$ and~\eqref{EQ:Green Func Def}, $\cL_z G(x,z)=\delta(z-x)$, so the right-hand side is $u(x)$.
\end{proof}
Note that for $d\ge2$, $G$ cannot be treated as a reproducing kernel in the sense of Definition~\ref{DEF:RKHS}, since $G(x,\cdot)\notin H^1(\Omega)$ (its gradient fails to
be square-integrable at the pole) and point evaluation is unbounded on $H_0^1(\Omega)$. Therefore, $(H_0^1,a)$ is an RKHS only in the generalized sense given in Theorem~\ref{THM:HG Gen}; see, for instance,~\cite{ArSm-AJM57}. This distinction does not affect the approximation framework. The direct bound uses $\cH_{\mathrm{src}}$ and $\cH_{\mathrm{sol}}$, while the global corrector bounds use the compatible input and output RKHSs of Corollary~\ref{COR:Global Rates}. In particular, the output space is a regularity-controlled subspace of $\cV_\cL$, not necessarily $\cV_\cL$ with its energy norm.

Let $\cV_\cL$ denote the space of finite-energy $\cL$-harmonic functions,
\[
   \cV_\cL:=\bigl\{\,w\in H^1(\Omega):\ \cL w=0\ \text{in }\Omega\,\bigr\} .
\]
The following energy orthogonality result is useful.
\begin{proposition}\label{PROP:Energy Orthogonality Gen}
For any $u\in H_0^1(\Omega)$ and $w\in\cV_\cL$, $a(u,w)=0$.
\end{proposition}
\begin{proof}
For $\phi\in C_c^\infty(\Omega)$, the identity $a(\phi,w)=0$ is precisely the weak form of $\cL w=0$: testing the equation against $\phi$ and integrating by parts produces no boundary term, because $\phi$ has compact support. Since $C_c^\infty(\Omega)$ is dense in $H_0^1(\Omega)$ and $a$ is continuous on $H^1(\Omega)\times H^1(\Omega)$ by~\eqref{EQ:Coeff Cond}, the identity extends to every $u\in H_0^1(\Omega)$. 
\end{proof}
Note that, for $w$ merely in $H^1(\Omega)$, the conormal derivative $\partial_\nu w$ is not a function, so the classical Green identity is not available. Therefore, we cannot prove the above result directly using Green's first identity.

\begin{theorem}\label{THM:HH Gen}
Assume $\sigma\not\equiv0$. The space $\cV_\cL$ of finite-energy
$\cL$-harmonic functions, equipped with the energy inner product~\eqref{EQ:Energy IP}, is a reproducing kernel Hilbert space on $\Omega$, whose elements are real-analytic in the interior. Analyticity up to the boundary requires additional regularity of their boundary data. Its reproducing kernel is the Neumann--minus--Dirichlet corrector
\begin{equation}\label{EQ:Reproducing_kernel}
   H^\cL_a(x,y):=N(x,y)-G(x,y),
\end{equation}
where $N$ is the Green's function of $\cL$ under the natural (conormal) boundary condition $\gamma\,\partial_\nu N(x,\cdot)|_{\partial\Omega}=0$. That is,
$H^\cL_a(x,\cdot)\in\cV_\cL$ for each $x$, and
\begin{equation}\label{EQ:Reproducing Gen}
   w(x)=a\bigl(w,\,H^\cL_a(x,\cdot)\bigr)\qquad\text{for all }w\in\cV_\cL .
\end{equation}
\end{theorem}
\begin{proof}
Recall ~\eqref{EQ:Energy IP} and standard coercivity and continuity arguments, we can easily get that $(\cV_\cL,a)$ is a Hilbert space where 
\[\cV_\cL:=\{w\in H^1(\Omega):a(w,\phi)=0\ \forall\phi\in C_c^\infty(\Omega)\}
\] is the intersection of kernels of the $H^1$-continuous functionals $w\mapsto a(w,\phi)$. Moreover, we have the norm equivalence $\|\cdot\|_a\simeq\|\cdot\|_{H^1}$ on $H^1(\Omega)$ with $a$ being a genuine inner product on the space. In particular, we let $C_{\mathrm{eq}}>0$ be the constant such that $\|\cdot\|_{H^1(\Omega)}\le C_{\mathrm{eq}}\|\cdot\|_a$.  
 
We now show that point evaluations are bounded, and elements of $(\cV_\cL,a)$ are analytic. A function $w\in\cV_\cL$ solves $\cL w=0$ with real-analytic coefficients, so by interior elliptic regularity $w\in \cC^\infty(\Omega)$, and by analytic regularity for elliptic operators with real-analytic coefficients~\cite{Morrey-Book66}, $w$ is real-analytic in $\Omega$. No uniform continuation across $\partial\Omega$ is asserted for the energy-space unit ball. For a fixed $x\in\Omega$, let us pick $r>0$ with $\overline{B(x,r)}\subset\Omega$.
The interior estimate for solutions of $\cL w=0$ supplies a constant
$C_{\mathrm{int}}=C_{\mathrm{int}}(x,r)$ with
\[
   |w(x)|\;\le\;C_{\mathrm{int}}\,\|w\|_{L^2(B(x,r))}
   \;\le\;C_{\mathrm{int}}\,\|w\|_{L^2(\Omega)}
   \;\le\;C_{\mathrm{int}}\,C_{\mathrm{eq}}\,\|w\|_a\,,
\]
where $C_{\mathrm{eq}}$ is the norm-equivalence constant of the first step
($\|\cdot\|_{L^2(\Omega)}\le\|\cdot\|_{H^1(\Omega)}\le C_{\mathrm{eq}}\|\cdot\|_a$). Thus $\delta_x$ is a bounded functional on $(\cV_\cL,a)$, which is therefore an RKHS.

Finally, the identification of the reproducing kernel can be established as follows. For $\eps>0$, let $\delta_x^\eps\in C_c^\infty(\Omega)$ be a standard mollification
of $\delta_x$, supported in $B(x,\eps)$ and with a unit mass. Since $\delta_x^\eps\in
(H^1(\Omega))^*$ and $a$ is coercive on $H^1(\Omega)$,
Lax--Milgram yields a unique $N^\eps(x,\cdot)\in H^1(\Omega)$ with
\[
a(N^\eps(x,\cdot),v)=\int_\Omega\delta_x^\eps v, \quad \forall\, v\in H^1(\Omega),
\]
i.e.\ $\cL N^\eps(x,\cdot)=\delta_x^\eps$ with the natural (conormal) boundary
condition $\gamma\,\partial_\nu N^\eps(x,\cdot)=0$. As $\eps\to 0$, classical potential theory gives $N^\eps(x,\cdot)\to N(x,\cdot)$, the Neumann-type Green's function, locally uniformly away from $x$. (Note that, for $d\ge 2$, $\delta_x\notin(H^1(\Omega))^*$). 
The self-adjointness of $\cL$ gives $N(x,y)=N(y,x)$. The function $H^\cL_a(x,\cdot)=N(x,\cdot)-G(x,\cdot)$ satisfies $\cL  H^\cL_a(x,\cdot)=\delta_x-\delta_x=0$, so
$H^\cL_a(x,\cdot)\in\cV_\cL$.
That is, with the two pole singularities canceling each other,  $H^\cL_a(x,\cdot)$ is in fact smooth in $\Omega$ and lies in $H^1(\Omega)$. For any $w\in\cV_\cL$ (and thus $w\in C^\infty(\Omega)$), with the Green's identity $a(w,v)=\int_\Omega w\,\cL v\,dz+\int_{\partial\Omega}\gamma\,w\,\partial_\nu v\,dS$ applied with $v=N^\eps(x,\cdot)$, we get 
\[
   a\bigl(w,N^\eps(x,\cdot)\bigr)
   =\int_\Omega w\,\delta^\eps_x\,dz+\int_{\partial\Omega}\gamma\,w\,
   \underbrace{\partial_\nu N^\eps(x,\cdot)}_{=\,0}\,dS \;\xrightarrow[\eps\to0]{}\;w(x).
\]  
The same mollified argument applied to $G$ gives $a(w,G^\eps(x,\cdot))=0$ for every $\eps$, since $G^\eps(x,\cdot)\in H_0^1(\Omega)$ (by Proposition~\ref{PROP:Energy Orthogonality Gen}). Moreover, the difference $N^\eps(x,\cdot)-G^\eps(x,\cdot)$ is $\cL$-harmonic, and its energy
\[
\|N^\eps(x,\cdot)-G^\eps(x,\cdot)\|_a^2
=\int_\Omega\delta_x^\eps\,\bigl(N^\eps(x,\cdot)-G^\eps(x,\cdot)\bigr)\,dz
\]
(with the cross term vanishing by the Proposition~\ref{PROP:Energy Orthogonality Gen}) remains bounded as $\eps\to0$: indeed, $N^\eps(x,\cdot)-G^\eps(x,\cdot)$ is $\cL$-harmonic in $\Omega$ with Dirichlet trace $N^\eps(x,\cdot)|_{\partial\Omega}$, which converges uniformly on $\partial\Omega$ as $\eps\to0$ for fixed interior $x$, so the weak maximum principle bounds $\|N^\eps(x,\cdot)-G^\eps(x,\cdot)\|_{L^\infty(\Omega)}$ uniformly in $\eps$, and the right-hand side above is bounded because $\delta_x^\eps$ has unit mass. Hence, $N^\eps(x,\cdot)-G^\eps(x,\cdot)$ converges weakly in $(H^1(\Omega),a)$ to $H^\cL_a(x,\cdot)$. Taking the difference of the two limits above then gives $a(w,H^\cL_a(x,\cdot))=w(x)$, which is~\eqref{EQ:Reproducing Gen}.
\end{proof}
 
The pointwise corrector $H^\cL=\Phi^\cL-G$ determines the target operator $T_{H^\cL}$ in the learning construction. The kernel $H^\cL_a=N-G$ in Theorem~\ref{THM:HH Gen} is the reproducing kernel of $(\cV_\cL,a)$, under the assumption $\sigma\not\equiv0$. The reproducing identity~\eqref{EQ:Reproducing Gen} applies to $H^\cL_a$ and does not follow for $H^\cL$ from the pointwise splitting alone. When $\gamma\equiv1$ and $\sigma\equiv0$, the energy form reduces to the Dirichlet inner product~\cite{ArSm-AJM57}, given by
\[
   \langle u,v\rangle_{\cD}:=\int_\Omega \nabla u\cdot\nabla v\,dx.
\]
Its induced norm $\|u\|_{\cD}=\|\nabla u\|_{L^2(\Omega)}$ is complete on harmonic functions modulo constants by the Poincar\'e inequality.
The distinction is genuine even in the pure-Laplacian case with $\gamma\equiv1$, $\sigma\equiv0$: on the unit disk, an explicit Fourier computation gives
$\langle w,H(x,\cdot)\rangle_{\cD}=\tfrac12\,w(x)$ for any harmonic $w$ (mod constants), i.e.\ the double-layer term is equal to $-\frac12w(x)$ up to a constant, and the reproducing kernel of $\cV_{\mathrm{harm}}/\bbR$ with respect to the Dirichlet inner product is $N-G$, which on the disk is equal to $2H$ modulo constants. Classical statements attributing the reproducing property to $H$ itself are therefore off by exactly this normalization. The construction of the harmonic kernel function from the Neumann and Green functions is classical; see, for instance, ~\cite{BeSc-Book53}.

\begin{theorem}\label{THM:Direct Sum Gen}
Assume $\sigma\not\equiv0$. Under the energy inner product~\eqref{EQ:Energy IP}, the space $H^1(\Omega)$ splits $a$-orthogonally into the zero-trace and $\cL$-harmonic subspaces, that is,
\begin{equation}\label{EQ:Direct Sum}
   \bigl(H^1(\Omega),a\bigr)
   =\bigl(H_0^1(\Omega),a\bigr)\ \oplus_a\ \bigl(\cV_\cL,a\bigr),
\end{equation}
and their energy reproducing kernels add, in the sense that
\begin{equation}\label{EQ:Kernel Sum}
   N=G+H^\cL_a,
   \qquad\text{i.e.}\qquad
   \cH_N\cong\cH_G\mathbin{\oplus_a}\cH_{H^\cL_a}
   \cong H_0^1(\Omega)\mathbin{\oplus_a}\cV_\cL ,
\end{equation}
where $N$ is the natural-boundary Green's function of Theorem~\ref{THM:HH Gen}, $H^\cL_a=N-G$, and $\cH_N$, $\cH_G$ denote the generalized energy spaces associated with $N$ and $G$, identified with $(H^1(\Omega),a)$ and $(H_0^1(\Omega),a)$, respectively, while $\cH_{H^\cL_a}$ denotes the RKHS generated by $H^\cL_a$, identified with $(\cV_\cL,a)$. 
Independently, the free-space fundamental solution $\Phi^\cL$ of $\cL$ (with
$\gamma,\sigma$ extended smoothly and uniformly elliptically to $\bbR^d$)
furnishes the \emph{pointwise} kernel splitting
\begin{equation}\label{EQ:Kernel-Decomp}
   G(x,y)=\Phi^\cL(x,y)-H^\cL(x,y),
   \qquad H^\cL(x,y):=\Phi^\cL(x,y)-G(x,y),
\end{equation}
in which $H^\cL(\cdot,y)\in\cV_\cL$ is the $\cL$-harmonic function with boundary
data $\Phi^\cL(\cdot,y)|_{\partial\Omega}$.
\end{theorem} 
\begin{proof}
As in the proof of Theorem~\ref{THM:HH Gen}, we know that $a(\cdot,\cdot)$ is an inner product on $H^1(\Omega)$ equivalent to the $H^1$ norm. For $w\in H^1(\Omega)$, the condition that $a(w,\phi)=0$ for all $\phi\in H_0^1(\Omega)$ is, by the weak formulation used in the proof of Proposition~\ref{PROP:Energy Orthogonality Gen}, equivalent to $\cL w=0$ weakly, i.e.\ $w\in\cV_\cL$. Hence, $\cV_\cL=(H_0^1)^{\perp_a}$ and~\eqref{EQ:Direct Sum} holds. Theorems~\ref{THM:HG Gen} and~\ref{THM:HH Gen} identify the (generalized) reproducing kernels of the two summands as $G$ and $H^\cL_a=N-G$. The identity $N=G+H^\cL_a$ in~\eqref{EQ:Kernel Sum} is then the definition~\eqref{EQ:Reproducing_kernel}, and the isomorphism in~\eqref{EQ:Kernel Sum} is the direct sum~\eqref{EQ:Direct Sum} just established. Finally, subtracting $\cL\Phi^\cL=\delta$ from $\cL G=\delta$ gives $\cL H^\cL=0$ in $\Omega$, while $G|_{\partial\Omega}=0$ gives $H^\cL|_{\partial\Omega}=\Phi^\cL|_{\partial\Omega}$,
which is~\eqref{EQ:Kernel-Decomp}.
\end{proof}
It is important to observe that, for $\sigma\equiv0$, every statement in Theorem~\ref{THM:Direct Sum Gen} holds after quotienting by constants, on which the energy form degenerates.

\paragraph{The Poisson special case.} When $\gamma\equiv1$ and $\sigma\equiv0$, the energy inner product is the Dirichlet inner product and the fundamental solution $\Phi^\cL=\Phi$ of~\eqref{EQ:Laplace Fundamental Sol} is explicit. The corresponding Dirichlet Green's function decomposes as
\begin{equation}\label{EQ:Laplace Green Func Decomp}
   G(x,y)=\Phi(x,y)-H(x,y),
\end{equation}
where, for each fixed $y\in\Omega$, the corrector $H(\cdot,y)$ is the unique
solution of
\begin{equation}\label{EQ:Laplace Harmonic Correction}
   \begin{cases}
      -\Delta_x H(x,y)=0, & x\in\Omega,\\
      H(x,y)=\Phi(x,y), & x\in\partial\Omega,
   \end{cases}
\end{equation}
and is smooth throughout $\overline\Omega$. At the function-space level, the corresponding energy decomposition is
\begin{equation}\label{eq:HPhi_decomp}
   H^1(\Omega)/\bbR
   \cong H_0^1(\Omega)\oplus_{\cD}\cV_{\mathrm{harm}},
\end{equation}
where $\cV_{\mathrm{harm}}:=\{w\in H^1(\Omega):-\Delta w=0\}/\bbR$, and $H_0^1(\Omega)$ is identified with its image in $H^1(\Omega)/\bbR$. The sum is orthogonal under the Dirichlet inner product, whose gradient norm is complete on these spaces. This energy decomposition is distinct from the pointwise identity $\Phi=G+H$, which by itself does not identify $\cH_\Phi$ with the energy space or establish equality of the associated inner products. For RKHS formulations of harmonic function spaces, see~\cite{Auchmuty-SIAM09,EnLuPePe-JM04}.

\subsection{Kernel decomposition for improved rates}
\label{SUBSEC:Improved Rates}

The direct bound~\eqref{EQ:Total Error Bound} treats $T$ as a generic operator between the spaces $\cH_{\mathrm{src}}$ and $\cH_{\mathrm{sol}}$. Theorem~\ref{THM:Direct Sum Gen} shows, however, that the singular element of the Green's function is carried entirely by the universal kernel $\Phi^\cL$, while everything the domain contributes lives in the analytic space $\cV_\cL$. We now embed this structure into a reformulation of the learning problem and account for the rates that it helps to gain. Throughout this subsection, $\Phi^\cL$ is used as an analytical device. Realizing the construction computably when $\Phi^\cL$ is not explicit is the task of Section~\ref{SEC:Computable-Decom}.

\begin{corollary}\label{COR:Exact Reduction}
For $f\in\cH_{\mathrm{src}}$ let $T_{\Phi^\cL}f:=\int_\Omega\Phi^\cL(\cdot,y)\,f(y)\,dy$ denote the free-space potential of $f$. Then $T_{\Phi^\cL}f$ splits $a$-orthogonally along Theorem~\ref{THM:Direct Sum Gen} as
\[
   T_{\Phi^\cL}f \;=\; Tf \;+\; w_f^\Phi,
   \qquad Tf\in\cH_G=(H_0^1(\Omega),a),\quad w_f^\Phi\in\cV_\cL ,
\]
where the corrector $w_f^\Phi(x)=\int_\Omega H^\cL(x,y)\,f(y)\,dy$ is the unique $\cL$-harmonic function with boundary data $(T_{\Phi^\cL}f)|_{\partial\Omega}$. Moreover, the solution operator factors as
\begin{equation}\label{EQ:Exact Operator Splitting}
   T \;=\; T_{\Phi^\cL}\;-\;E_\cL\circ\mathfrak{R}\circ T_{\Phi^\cL}\,,
\end{equation}
where $\mathfrak{R}$ is the restriction to $\partial\Omega$ and $E_\cL:h\mapsto w$ ($\cL w=0\ \mbox{in}\ \Omega$, $w|_{\partial\Omega}=h$) is the $\cL$-harmonic extension operator.
\end{corollary}
\begin{proof}
By the kernel splitting~\eqref{EQ:Kernel-Decomp} and linearity, $T_{\Phi^\cL}f-Tf=\int_\Omega H^\cL(\cdot,y)f(y)\,dy=:w_f^\Phi$. Each $H^\cL(\cdot,y)$ is $\cL$-harmonic (Theorem~\ref{THM:Direct Sum Gen}), so $\cL w_f^\Phi=0$ in $\Omega$, and $w_f^\Phi|_{\partial\Omega}=(T_{\Phi^\cL}f)|_{\partial\Omega}$ because $Tf|_{\partial\Omega}=0$. Hence $w_f^\Phi=E_\cL[\mathfrak{R}\,T_{\Phi^\cL}f]$, which is~\eqref{EQ:Exact Operator Splitting}. The $a$-orthogonality of the two summands is Proposition~\ref{PROP:Energy Orthogonality Gen}.
\end{proof}

In~\eqref{EQ:Exact Operator Splitting} the potential $T_{\Phi^\cL}$ and the restriction $\mathfrak{R}$ are known. The only unknown is the $\cL$-harmonic extension $E_\cL$. Learning $T$ is therefore equivalent to the following reformulated problem.

\medskip
\noindent 
\textbf{Reduced learning problem (A):}\ \ Learn the $\cL$-harmonic extension $E_\cL:h\mapsto w$ from the re-processed data $\{(h_\ell,\,w_\ell)\}_{\ell=1}^N$, where $h_\ell:=\mathfrak{R}\,T_{\Phi^\cL}f_\ell$ and $w_\ell:=T_{\Phi^\cL}f_\ell-u_\ell$. Then recover $Tf=T_{\Phi^\cL}f-E_\cL\bigl[\mathfrak{R}\,T_{\Phi^\cL}f\bigr]$. 
\medskip

The reformulation replaces the operator $T:\cH_{\mathrm{src}}\to\cH_{\mathrm{sol}}$
by $E_\cL:\cH_\partial^\cL\to\cV_\cL$, with $\cH_\partial^\cL$ the boundary trace space identified below. It improves the input rate, while additional regularity of the boundary data permits faster output convergence in a suitably chosen output RKHS.

We first identify the boundary-input regularity, then state the global rates in a separate corollary. The output RKHS will be a regularity-controlled subspace of $\cV_\cL$. It does not need to carry the energy norm.

The input of $E_\cL$ is the trace $h=\mathfrak{R}\,T_{\Phi^\cL}f$.  The potential solves $\cL(T_{\Phi^\cL}f)=f$ in $\Omega$. With the coefficient extension of Theorem~\ref{THM:Direct Sum Gen}, the one-sided regularity of volume potentials gives $\|T_{\Phi^\cL}f\|_{H^{s+2}(\Omega)}\le C\|f\|_{H^s(\Omega)}$~\cite{McLean-Book00}. No such cross-boundary Sobolev estimate is asserted for the zero-extended source. The trace theorem therefore gives $h\in H^{s+3/2}(\partial\Omega)$. The induced trace space satisfies the continuous embedding
\[
   \cH_\partial^\cL:=\bigl\{\mathfrak{R}\,T_{\Phi^\cL}f\ :\ f\in H^s(\Omega)\bigr\}
   \;\hookrightarrow\;H^{s+3/2}(\partial\Omega)
\]
under the minimum-source norm $\|h\|_{\cH_\partial^\cL}:=\min\{\|f\|_{H^s(\Omega)}:\mathfrak R T_{\Phi^\cL}f=h\}$. It is a complete image RKHS. Surjectivity onto the full Sobolev trace space is not needed. Theorem~\ref{THM:H Partial} gives the corresponding construction for the computable potential. The Sobolev sampling inequality on the $(d-1)$-dimensional boundary, applied to the kernel projection residual, then gives~\cite{FuWr-SINUM12}
\begin{equation}\label{EQ:Power Partial Exact}
   \sup_{\omega\in\partial\Omega}P_{\cH_\partial^\cL,W}(\omega)
   \;\le\;C\,h_W^{\,s+3/2-\frac{d-1}{2}}
   \;=\;C\,h_W^{\,s+2-\frac{d}{2}} ,
\end{equation}
two full orders above the volume rate $\cO(h_W^{\,s-d/2})$ of the direct input: the potential lifts the Sobolev order by $2$, the trace costs $1/2$, and the drop in dimension from $d$ to $d-1$ contributes the remaining $1/2$.

We summarize the observations into the following corollary, in which $\eps_Y$ denotes an arbitrary bound for the output power function. The two concrete choices of $\eps_Y$ used later are stated as parts (i) and (ii). The assumptions behind part (ii) are discussed in Remark~\ref{REM:Rate Regimes} below.
\begin{corollary}\label{COR:Global Rates}
Let $s>d/2$, and let $\cH_X$ and $\cH_Y$ be RKHSs satisfying
\[
   \cH_X\hookrightarrow H^{s+3/2}(\partial\Omega),
   \qquad \cH_Y\subset\cV_\cL,
\]
such that $E_\cL:\cH_X\to\cH_Y$ is bounded. Let $P_X,P_Y$ be their kernel interpolation projections on the chosen sampling sets, with invertible Gram matrices and fill distances $h_X$ and $h_Y$, and set $\wh E_\cL=P_YE_\cL P_X$, represented by the ideal matrix $\bM^\star$ of~\eqref{EQ:GStar}. Let $\eps_Y(h_Y)$ be any bound for the output power function, $\sup_{x\in\overline\Omega}P_{K_Y,V}(x)\le\eps_Y(h_Y)$. Then, for every $h\in\cH_X$ with $\|h\|_{\cH_X}\le R_\partial$,
\begin{equation}\label{EQ:Total Error Exact}
   \|E_\cL h-\wh E_\cL h\|_{L^\infty(\Omega)}
   \le M_{\cV}\,\eps_Y(h_Y)
   +L_E R_\partial\,\widetilde C_X h_X^{\,s+2-d/2},
\end{equation}
where $M_{\cV}:=R_\partial\|E_\cL\|_{\cH_X\to\cH_Y}$, $L_E\le1$ is the constant of Lemma~\ref{LEM:Harmonic Contraction}, and $\widetilde C_X$ is the constant of~\eqref{EQ:Power Partial Exact}. In particular:
\begin{enumerate}[label=(\roman*)]
\item \emph{Sobolev regime.} If $\cH_Y\hookrightarrow H^t(\Omega)$ for some $t\ge s+2$, then~\eqref{EQ:Total Error Exact} holds with $\eps_Y(h_Y)=\widetilde C_Y h_Y^{\,t-d/2}$, that is,
\[
   \|E_\cL h-\wh E_\cL h\|_{L^\infty(\Omega)}
   \le M_{\cV}\widetilde C_Y\,h_Y^{\,t-d/2}
   +L_E R_\partial\,\widetilde C_X h_X^{\,s+2-d/2}.
\]
\item \emph{Analytic regime.} If the output kernel $K_Y$ and the sampling family $V$ satisfy the exponential power-function estimate
\begin{equation}\label{EQ:Power Exp Global}
   \sup_{x\in\overline\Omega}P_{K_Y,V}(x)\le\widetilde C_Y\,e^{-c_Y/h_Y}
\end{equation}
with constants $\widetilde C_Y,c_Y>0$ independent of $h_Y$, then~\eqref{EQ:Total Error Exact} holds with $\eps_Y(h_Y)=\widetilde C_Y e^{-c_Y/h_Y}$, that is,
\[
   \|E_\cL h-\wh E_\cL h\|_{L^\infty(\Omega)}
   \le M_{\cV}\widetilde C_Y\,e^{-c_Y/h_Y}
   +L_E R_\partial\,\widetilde C_X h_X^{\,s+2-d/2}.
\]
\end{enumerate}
\end{corollary}
\begin{proof}
The input embedding gives the estimate~\eqref{EQ:Power Partial Exact} with $\cH_X$ in place of $\cH_\partial^\cL$. We write
\[
   E_\cL h-P_YE_\cL P_Xh
   =E_\cL(h-P_Xh)+(I-P_Y)E_\cL P_Xh.
\]
The maximum principle bounds the first term by $L_E R_\partial\widetilde C_Xh_X^{s+2-d/2}$ (Lemma~\ref{LEM:Harmonic Contraction}). The second is at most $\eps_Y(h_Y)\|E_\cL P_Xh\|_{\cH_Y}\le M_{\cV}\eps_Y(h_Y)$ because $\|P_Xh\|_{\cH_X}\le R_\partial$ (Theorem~\ref{THM:Power Func}), which proves~\eqref{EQ:Total Error Exact}. In case (i), Theorem~\ref{THM:Power Func Decay}(i) applied to $\cH_Y\hookrightarrow H^t(\Omega)$ shows that $\eps_Y(h_Y)=\widetilde C_Yh_Y^{\,t-d/2}$ is an admissible bound. In case (ii), the admissible bound $\eps_Y(h_Y)=\widetilde C_Ye^{-c_Y/h_Y}$ is the hypothesis~\eqref{EQ:Power Exp Global} itself.
\end{proof}

\begin{remark}\label{REM:Rate Regimes} 
Part (ii) is conditional: the exponential decay of the output term is not derived but assumed through~\eqref{EQ:Power Exp Global}, which is a property of the pair $(K_Y,V)$. The regime is called analytic because an estimate of the form~\eqref{EQ:Power Exp Global} can be expected only when the unit ball of $\cH_Y$ consists of uniformly analytic functions; analytic geometry and analytic coefficients alone do not imply it. The natural sufficient setting is the following. Assume that $\partial\Omega$, $\gamma$ and $\sigma$ are analytic, and that every $w\in\cH_Y$ extends analytically to a function $\bar w$ on a fixed neighborhood $D\Supset\overline\Omega$ with the Cauchy-type bound
\[
   \|D^\alpha\bar w\|_{L^\infty(D)}
   \le A\rho^{-|\alpha|}\,|\alpha|!\,\|w\|_{\cH_Y}
   \quad\text{for every multi-index }\alpha,
\]
where $A,\rho>0$ are independent of $w$. For analytic regularity up to the boundary with analytic data, see~\cite{MoNi-CPAM57} and~\cite[Proposition~2.1]{CoDaNi-M3AS12}. Under this condition the output class is uniformly analytic, which is the situation in which exponential kernel-interpolation estimates are known for analytic kernels (compare Theorem~\ref{THM:Power Func Decay}(ii) for the Gaussian kernel).
\end{remark}

\section{Computable decomposition via Laplace parametrix}
\label{SEC:Computable-Decom}

The single obstruction to adopting the exact splitting~\eqref{EQ:Kernel-Decomp} as a learning strategy is that, unlike the Laplace fundamental solution $\Phi$, the operator $\Phi^\cL$ is not available in closed form for variable $\gamma$. We now develop a computable decomposition via the standard Laplace parametrix. The construction is rooted in classical potential theory and boundary integral methods, where elliptic boundary-value problems are represented through volume and layer potentials~\cite{McLean-Book00}. Locally approximating the fundamental solution by freezing the principal coefficient is the classical Levi parametrix method~\cite{Levi-RCMP07,Miranda-Book70}. Parametrix-based boundary-domain formulations have also been developed for variable-coefficient elliptic problems~\cite{Mikhailov-EABE02}. A data-driven alternative learns the domain-independent free-space kernel of a variable-coefficient operator numerically~\cite{NeChKrYiLi-CMAME24}. Here, that kernel is replaced by an explicit parametrix. Our use of this construction is different: we employ it to reformulate the operator-learning data and analyze the resulting approximation rates.

\subsection{Decomposition via Laplace parametrix}
\label{SUBSEC:Laplace Para}

Let us start by freezing the principal coefficient at the source point and replacing $\Phi^\cL$ by the explicit Laplace parametrix
\begin{equation}\label{EQ:P0}
   P_0(x,y):=\frac{1}{\gamma(y)}\,\Phi(x-y),
\end{equation}
which carries exactly the same principal singularity as $\Phi^\cL$. A direct distributional computation, using $\nabla_x P_0=\gamma(y)^{-1}\nabla\Phi(x-y)$, $\Delta\Phi(x-y)=-\delta(x-y)$, and $\gamma(x)\delta(x-y)=\gamma(y)\delta(x-y)$, gives
\begin{equation}\label{EQ:LP0}
   \cL_x P_0(x,y)=\delta(x-y)+F_0(x,y),
   \qquad
   F_0(x,y)=-\frac{\nabla\gamma(x)\cdot\nabla\Phi(x-y)}{\gamma(y)}
   +\frac{\sigma(x)\,\Phi(x-y)}{\gamma(y)}.
\end{equation}
The term $F_0$, referred to as the defect, is the price of using the Laplace fundamental solution for an operator with a variable $\gamma$ and a zeroth-order term. It splits into a piece proportional to $\nabla\gamma$, singular of order $|x-y|^{1-d}$, and a strictly milder piece proportional to $\sigma$, singular only of order $|x-y|^{2-d}$ (logarithmic when $d=2$). The defect vanishes identically when $\gamma$ is constant and $\sigma\equiv0$, in which case $P_0=\Phi$ is the exact fundamental solution and we are back in the Poisson decomposition~\eqref{EQ:Laplace Green Func Decomp}.

Let us now define the computable corrector:
\begin{equation}
   H_0(x,y):=P_0(x,y)-G(x,y),
   \qquad\text{so that}\qquad G=P_0-H_0 .
\end{equation}
It is straightforward to verify that $H_0$ is $\cL$-harmonic up to the defect in the sense that:
\begin{equation}\label{EQ:H0 Equation}
   \cL_x H_0(x,y)=F_0(x,y)\ \text{in }\Omega,
   \qquad H_0(x,y)=P_0(x,y)\ \text{on }\partial\Omega,
\end{equation}
so it splits as $H_0(\cdot,y)=w(\cdot,y)+p(\cdot,y)$ with 
\[\cL w=0,\; w|_{\partial\Omega}=P_0(\cdot,y)|_{\partial\Omega},\; \text{ and }\; \cL p=F_0(\cdot,y), \; p|_{\partial\Omega}=0\]
where $p$ is the defect potential. When $\gamma$ is constant and $\sigma\equiv0$, $F_0\equiv0$, so $p\equiv0$ and $H_0=H$ is exactly harmonic:~\eqref{EQ:H0 Equation} collapses to~\eqref{EQ:Laplace Green Func Decomp}. For variable coefficients, the only new object is the smooth defect potential $p$, and it is benign: since $|\nabla\Phi(z)|\sim|z|^{1-d}$, the leading part of $F_0(\cdot,y)$ lies in $L^r_{\mathrm{loc}}$ for $r<d/(d-1)$. The interior elliptic regularity makes $p(\cdot,y)$ more regular than $G$ itself near the pole.

The leading $\nabla\gamma$-singularity of $F_0$ can be canceled by appending the explicit first-order term
\[
   P_1(x,y):=-\frac{\nabla\gamma(y)\cdot(x-y)}{2\,\gamma(y)^2}\,\Phi(x-y),
\]
which leaves a defect one power less singular. The cancellation mechanism is elementary: for any fixed vector $a\in\bbR^d$ one has $\Delta\bigl((a\cdot z)\Phi(z)\bigr)=2\,a\cdot\nabla\Phi(z)$ for $z\neq0$, so, freezing the coefficients at the source point,
\[
   -\gamma(y)\,\Delta_x P_1(x,y)
   =\frac{\nabla\gamma(y)\cdot\nabla\Phi(x-y)}{\gamma(y)},
   \qquad x\neq y,
\]
which is the exact negative of the coefficient-frozen leading part of $F_0$ in~\eqref{EQ:LP0}. Taylor-expanding $\gamma$ at $y$ in the remaining terms shows that the new defect $F_1:=\cL_x (P_0+P_1)(\cdot,y)-\delta(\cdot-y)$ is of order $\cO(|x-y|\,|\nabla\Phi(x-y)|)=\cO(|x-y|^{2-d})$.

The construction extends to arbitrary order $m\ge1$ by the classical Levi iteration~\cite{Friedman-Book08}, which we record explicitly. Define the iterated defect kernels
\[
   F_0^{\ast1}:=F_0,
   \qquad
   F_0^{\ast(j+1)}(x,y):=\int_\Omega F_0(x,z)\,F_0^{\ast j}(z,y)\,dz,
   \quad j\ge1,
\]
and the $m$-th order parametrix
\begin{equation}\label{EQ:Pm}
   P^{(m)}(x,y):=P_0(x,y)
   +\sum_{j=1}^{m}(-1)^j\int_\Omega P_0(x,z)\,F_0^{\ast j}(z,y)\,dz .
\end{equation}
Applying~\eqref{EQ:LP0} to each term of~\eqref{EQ:Pm} telescopes the defects:
\begin{equation}\label{EQ:Fm}
   \cL_x P^{(m)}(x,y)=\delta(x-y)+F_m(x,y),
   \qquad
   F_m:=(-1)^m F_0^{\ast(m+1)} .
\end{equation}
We note that each composition with the weakly singular kernel $F_0=\cO(|x-y|^{1-d})$ raises the order of the singularity by one, so $F_m=\cO(|x-y|^{m+1-d})$ (bounded, up to logarithmic factors, once $m\ge d-1$). That is, each order of the expansion buys exactly one power of the defect, consistent with the first-order computation above. 

Two practical points are worth noting. First, every term in~\eqref{EQ:Pm} is an explicit iterated integral of the \emph{known} kernels $P_0$ and $F_0$, evaluable by quadrature. 
There is no requirement to solve the problem with the operator $\cL$, so $P^{(m)}$ is computable in exactly the same sense as $P_0$. Second, at low order, the iterated integrals may be traded for closed-form Taylor corrections: carrying the Taylor expansion of $\gamma$ and $\sigma$ at $y$ to the order $m$ and appending polynomial-modulated terms $q_j(x-y;y)\,\Phi(x-y)$, with $q_j$ being a polynomial of degree $j$ in $x-y$ (of which $P_1$ above is the degree-one instance), realizes the same order-by-order reduction. The recursion~\eqref{EQ:Pm} represents a systematic formulation. The effect of $P^{(m)}$ on the learning rates is quantified in the following subsection.

\subsection{The decomposed learning strategy and its rates}
\label{SUBSEC:Improved Algo}

The parametrix is explicit, so the \emph{parametrix potential}
\begin{equation}\label{eq:TP0_gen}
   (T_{P_0}f)(x):=\int_\Omega P_0(x,y)\,f(y)\,dy
   =\bigl(\Phi*(\gamma^{-1}f)\bigr)(x)
\end{equation}
is computable for any $f$, which for $\gamma\equiv1$ is the Newtonian potential $T_\Phi f=\int_\Omega\Phi(x,y)f(y)\,dy$. We define the corrector field \[
    w_f:=T_{P_0}f-Tf\,.
\]
From here on $w_f$ denotes the parametrix corrector, whose harmonic and defect-potential parts are separated below. It differs from the exact corrector $w^{\Phi}_f$ of Corollary~\ref{COR:Exact Reduction}. The two coincide when $P_0=\Phi^\cL$. 

Using $G=P_0-H_0$, we have $w_f(x)=\int_\Omega H_0(x,y)\,f(y)\,dy$. Applying $\cL$ through~\eqref{EQ:LP0} shows that $w_f$ solves the fully explicit boundary-value problem
\begin{equation}\label{EQ:wf BVP}
   \cL\,w_f=\mathcal{F}_0 f\ \text{in }\Omega,
   \qquad w_f=\bigl(T_{P_0}f\bigr)\big|_{\partial\Omega}\ \text{on }\partial\Omega,
\end{equation}
where $\mathcal{F}_0:=\cL\,T_{P_0}-I$ is the defect operator and we used $Tf|_{\partial\Omega}=0$. Both the forcing and the boundary data are computable from $f$ and the known $\Phi,\gamma,\sigma$. Splitting as in~\eqref{EQ:H0 Equation},
\begin{equation}\label{EQ:Wf Split}
w_f=\underbrace{E_\cL\bigl[(T_{P_0}f)|_{\partial\Omega}\bigr]}_{\cL\text{-harmonic extension}}
   \;+\;\underbrace{S_\cL\bigl[\mathcal{F}_0 f\bigr]}_{\text{defect potential}},
\end{equation}
where $E_\cL:h\mapsto(\cL\text{-harmonic }w,\ w|_{\partial\Omega}=h)$ is the $\cL$-harmonic extension operator and $S_\cL:=\cL^{-1}$ with zero boundary data. For $\gamma\equiv1,\sigma\equiv0$ one has $\mathcal{F}_0\equiv0$, $E_\cL=E$ the harmonic extension, and $w_f=E[(T_\Phi f)|_{\partial\Omega}]$. To be more concrete, for each pair $(f_\ell,u_\ell)$, the boundary data $h_\ell:=(T_\Phi f_\ell)|_{\partial\Omega}$ and interior corrector $w_\ell:=(T_\Phi f_\ell)-u_\ell$ satisfy the harmonic extension system
\begin{equation}
   \Delta w_\ell=0\ \text{in }\Omega,\qquad w_\ell=h_\ell\ \text{on }\partial\Omega,
\end{equation}
and $T$ is represented as $T=T_\Phi-E\circ\mathfrak{R}\circ T_\Phi$ with $\mathfrak{R}$ the restriction to $\partial\Omega$.

\medskip
\noindent
\textbf{Reduced learning problem (B):}\ \ Learn $f\mapsto w_f$ from the re-expressed
data $\{(f_\ell,\,w_{f_\ell})\}_{\ell=1}^N$, $w_{f_\ell}=T_{P_0}f_\ell-u_\ell$, then recover
$Tf=T_{P_0}f-w_f$. The unknown element is the $\cL$-harmonic extension $E_\cL$ plus
the smooth defect potential $S_\cL\circ\mathcal{F}_0$.
\medskip

The harmonic-extension term in~\eqref{EQ:Wf Split} is treated by Corollary~\ref{COR:Global Rates}, subject to its input--output compatibility and regularity assumptions. Its input is the boundary trace $(T_{P_0}f)|_{\partial\Omega}$, whose Sobolev regularity follows from the next result.
 
\begin{theorem}\label{THM:H Partial}
Let $s>d/2$, and assume sufficient boundary smoothness for the volume-potential and trace estimates below. Define $\Lambda f=(\Phi*f)|_{\partial\Omega}$ and $\cH_\partial:=\Lambda(H^s(\Omega))$, equipped with the minimum-source norm
\[
\|h\|_{\cH_\partial}:=\min\{\|f\|_{H^s(\Omega)}:\Lambda f=h\}.
\]
Then $\cH_\partial$ is an RKHS continuously embedded in $H^{s+3/2}(\partial\Omega)$, with reproducing kernel
\[
   K_\partial(x,x')=\int_\Omega\int_\Omega\Phi(x,y)\,K_s(y,y')\,\Phi(x',y')\,dy\,dy',
   \qquad x,x'\in\partial\Omega,
\]
where $K_s$ is the reproducing kernel of the chosen Sobolev RKHS $H^s(\Omega)$. \end{theorem}
\begin{proof}
The one-sided regularity of the Newtonian volume potential~\cite{McLean-Book00} and the trace theorem~\cite[Section 5.5]{Evans-Book10} give
\[
\|\Lambda f\|_{H^{s+3/2}(\partial\Omega)}
\le C\|\Phi*f\|_{H^{s+2}(\Omega)}
\le C'\|f\|_{H^s(\Omega)}.
\]
Thus $\ker\Lambda$ is closed. Each preimage has a unique minimum-norm representative in $(\ker\Lambda)^\perp$, so $\cH_\partial$ is isometric to the complete quotient $H^s(\Omega)/\ker\Lambda$. Taking the minimum over preimages proves the claimed continuous embedding. Since $s+3/2>(d-1)/2$, point evaluation on $\cH_\partial$ is bounded. Applying $\Lambda$ in each argument of $K_s$ gives the displayed reproducing kernel~\cite{SaSa-Book16}.
\end{proof}

For the weighted potential $T_{P_0}$, the same result applies to $\Lambda(\gamma^{-1}f)$: multiplication by the smooth, bounded function $\gamma^{-1}$ is bounded on $H^s(\Omega)$. The continuous embedding, together with the Sobolev zeros estimate on $\partial\Omega$ used in Corollary~\ref{COR:Global Rates}, gives
\begin{equation}
   \sup_{\omega\in\partial\Omega}P_{\cH_\partial,W}(\omega)
   \le C\,h_W^{\,s+3/2-\frac{d-1}{2}},
\end{equation}
where $d-1$ is the dimension of $\partial\Omega$ and $W\subset\partial\Omega$ is the input sampling set ($h_W=h_X$ in this subsection). The exponent exceeds the direct $H^s(\Omega)$ rate by $2$, reflecting the smoothness gained by applying the Newtonian potential before restricting to the boundary. Besides this input-regularity gain, the extension operator is stable in $L^\infty$.

More precisely, it is an $L^\infty$-contraction, as recorded below.
\begin{lemma}\label{LEM:Harmonic Contraction} 
The extension operator $E_\cL$, sending boundary data $h\in C(\partial\Omega)$ to the solution $w$ of $\cL w=0$, $w|_{\partial\Omega}=h$, is linear and satisfies
$\|E_\cL h\|_{L^\infty(\Omega)}\le\|h\|_{L^\infty(\partial\Omega)}$, which gives the $L^\infty$-Lipschitz part of assumption~\hyperlink{Assum-A2}{\Aii} with constant $L_E\le1$. Boundedness from $\cH_X$ into a chosen output RKHS $\cH_Y$ is a separate assumption of Corollary~\ref{COR:Global Rates}.
\end{lemma}
\begin{proof}
Linearity is immediate. Since $w=E_\cL h$ solves $\cL w=0$ with $\sigma\ge0$, the weak maximum principle for $\cL$ gives $\|w\|_{L^\infty(\Omega)}\le\|h\|_{L^\infty(\partial\Omega)}$, with equality when $\sigma\equiv0$. Applying this to $h_1-h_2$ and using linearity yields the Lipschitz bound.
\end{proof}

If we choose compatible input and output RKHSs as in Corollary~\ref{COR:Global Rates}, and assume the induced boundary-data class is bounded by $R_\partial$ in $\cH_X$, then for the ideal projected extension operator, that corollary gives
\begin{equation}\label{EQ:Total Error Corrector Form}
   \|E_\cL(h)-\wh E_\cL(h)\|_{L^\infty(\Omega)}
   \le M_{\cV}\,\eps_Y(h_Y)
   + L_E R_\partial\,\widetilde C_X\,h_X^{\,s+2-d/2},
\end{equation}
where $M_{\cV}=R_\partial\|E_\cL\|_{\cH_X\to\cH_Y}$. For the Sobolev regime one may take $\cH_X=\cH_\partial$ and $\cH_Y=\cV_\cL\cap H^{s+2}(\Omega)$ with its $H^{s+2}$ norm. Theorem~\ref{THM:H Partial} then supplies $R_\partial$ from a bounded source class ($R_\partial\le R$ for the ball of radius $R$, by the minimum-source norm), and $\eps_Y(h_Y)\le C_Yh_Y^{s+2-d/2}$. Higher output regularity improves this algebraic exponent; the exponential choice in~\eqref{EQ:Power Exp Global} requires the analytic regime of Corollary~\ref{COR:Global Rates}(ii), not merely analytic coefficients and boundary. The defect potential is estimated separately.

\begin{proposition}\label{PROP:Defect}
The defect operator $\mathcal{F}_0=\cL\,T_{P_0}-I$ of~\eqref{EQ:wf BVP} satisfies the pointwise identity
\begin{equation}\label{EQ:F0 Identity}
   \mathcal{F}_0 f=-\,\nabla\gamma\cdot\nabla\bigl(T_{P_0}f\bigr)+\sigma\,T_{P_0}f,
\end{equation}
and $\mathcal{F}_0:H^s(\Omega)\to H^{s+1}(\Omega)$ is bounded. Consequently the defect potential $p_f:=S_\cL\mathcal{F}_0 f=T\mathcal{F}_0 f$ satisfies $p_f\in H^{s+3}(\Omega)\cap H_0^1(\Omega)$ with $\|p_f\|_{H^{s+3}(\Omega)}\lesssim\|f\|_{H^s(\Omega)}$, i.e.\ it is \emph{one Sobolev order smoother} than the solution $Tf\in H^{s+2}(\Omega)$.
\end{proposition}
\begin{proof}
Since $T_{P_0}f=\Phi*(\gamma^{-1}f)$ satisfies $-\Delta(T_{P_0}f)=\gamma^{-1}f$,
\[
   \cL\,T_{P_0}f=-\nabla\gamma\cdot\nabla(T_{P_0}f)-\gamma\,\Delta(T_{P_0}f)+\sigma\,T_{P_0}f
   =f-\nabla\gamma\cdot\nabla(T_{P_0}f)+\sigma\,T_{P_0}f,
\]
which is~\eqref{EQ:F0 Identity} after subtracting $f$. The weighted Newtonian potential gains two derivatives, $T_{P_0}:H^s\to H^{s+2}$, so $\nabla(T_{P_0}f)\in H^{s+1}$ and $T_{P_0}f\in H^{s+2}\subset H^{s+1}$. Multiplication by the smooth fields $\nabla\gamma,\sigma$ preserves $H^{s+1}$, giving $\mathcal{F}_0:H^s\to H^{s+1}$.

Finally $p_f=T\mathcal{F}_0 f$ has forcing in $H^{s+1}$ and zero Dirichlet data, so the elliptic regularity yields $p_f\in H^{s+3}\cap H_0^1$.
\end{proof}

The $m$-th order parametrix of Section~\ref{SUBSEC:Laplace Para} increases the Sobolev regularity of the defect potential by one order for each added parametrix term, without otherwise changing the learning strategy: one simply replaces $T_{P_0}$ by $T_{P^{(m)}}f:=\int_\Omega P^{(m)}(\cdot,y)f(y)\,dy$ in the definition of the corrector field and in the re-expressed data $w^{(m)}_{f_\ell}:=T_{P^{(m)}}f_\ell-u_\ell$.

\begin{corollary}\label{COR:Defect-m}
Let $P^{(m)}$ be the $m$-th order parametrix~\eqref{EQ:Pm} and $\mathcal{F}_m:=\cL\,T_{P^{(m)}}-I$ the associated defect operator. Then
\begin{equation}\label{EQ:Fm-operator}
   \mathcal{F}_m=(-1)^m\,\mathcal{F}_0^{\,m+1},
\end{equation}
so $\mathcal{F}_m:H^s(\Omega)\to H^{s+m+1}(\Omega)$ is bounded, and the corresponding defect potential $p^{(m)}_f:=S_\cL\mathcal{F}_m f=T\mathcal{F}_m f$
satisfies
\[
   p^{(m)}_f\in H^{s+2+(m+1)}(\Omega)\cap H_0^1(\Omega),
   \qquad
   \|p^{(m)}_f\|_{H^{s+2+(m+1)}(\Omega)}\lesssim\|f\|_{H^s(\Omega)},
\]
i.e.\ it is $m+1$ Sobolev orders smoother than the solution $Tf\in H^{s+2}(\Omega)$.
\end{corollary}
\begin{proof}
By~\eqref{EQ:LP0} the Schwartz kernel of $\mathcal{F}_0$ is $F_0$, so the kernel of the operator power $\mathcal{F}_0^{\,j}$ is the iterated kernel $F_0^{\ast j}$, and the $j$-th term of~\eqref{EQ:Pm} is the kernel of $T_{P_0}\mathcal{F}_0^{\,j}$. Hence,   $T_{P^{(m)}}=T_{P_0}\sum_{j=0}^{m}(-1)^j\mathcal{F}_0^{\,j}$. Applying $\cL\,T_{P_0}=I+\mathcal{F}_0$ termwise telescopes:
\[
   \cL\,T_{P^{(m)}}
   =(I+\mathcal{F}_0)\sum_{j=0}^{m}(-1)^j\mathcal{F}_0^{\,j}
   =\sum_{j=0}^{m}(-1)^j\bigl(\mathcal{F}_0^{\,j}+\mathcal{F}_0^{\,j+1}\bigr)
   =I+(-1)^m\mathcal{F}_0^{\,m+1},
\]
which is~\eqref{EQ:Fm-operator}. This is the operator form of the kernel identity~\eqref{EQ:Fm}. The proof of Proposition~\ref{PROP:Defect} applies at every intermediate smoothness level (the weighted Newtonian potential smooths by two orders on each $H^{s'}$, $s'\ge s$, and multiplication by the smooth fields $\nabla\gamma,\sigma$ preserves $H^{s'+1}$), so each of the $m+1$ factors of $\mathcal{F}_0$ in~\eqref{EQ:Fm-operator} gains one order and $\mathcal{F}_m:H^s\to H^{s+m+1}$. 

Finally, $p^{(m)}_f=T\mathcal{F}_m f$ has forcing in $H^{s+m+1}$ and zero Dirichlet data, so elliptic regularity yields $p^{(m)}_f\in H^{s+m+3}\cap H_0^1=H^{s+2+(m+1)}\cap H_0^1$.
\end{proof}

Applying the projection estimates separately to the two terms of~\eqref{EQ:Wf Split} gives Table~\ref{TAB:Rate Comparison}. The harmonic part has output factor $\eps_Y(h_Y)$ and boundary-input rate $\cO(h_X^{s+2-d/2})$. Its baseline output rate is $\cO(h_Y^{s+2-d/2})$, improved under the additional regularity assumptions of Corollary~\ref{COR:Global Rates}. The defect potential has the higher output rate $\cO(h_Y^{s+3-d/2})$ by Proposition~\ref{PROP:Defect}, but retains the volume-input rate $\cO(h_X^{s-d/2})$. For each fixed $m$, Corollary~\ref{COR:Defect-m} further improves the defect output exponent to $s+m+3-d/2$, with constants that may depend on $m$. These regularity gains alone imply neither convergence of the parametrix series as $m\to\infty$ nor a smaller total error for every sampling configuration.
\begin{table}[!htb]
\centering
\renewcommand{\arraystretch}{1.4}
\resizebox{\textwidth}{!}{%
\begin{tabular}{lccc}
\hline
 & \textbf{Direct} & \textbf{Decomposed: $\cL$-harmonic} & \textbf{Decomposed: defect} \\
\hline
Target operator              & $T:\cH_{\mathrm{src}}\to\cH_{\mathrm{sol}}$ & $E_\cL:\cH_X\to\cH_Y$ & $S_\cL\mathcal{F}_0:\cH_{\mathrm{src}}\to H^{s+3}(\Omega)\cap H_0^1(\Omega)$ \\
Training data                & $\{(f_\ell,u_\ell)\}_{\ell=1}^N$ & \multicolumn{2}{c}{$\{(f_\ell,\,w_{f_\ell})\}_{\ell=1}^N,\quad w_{f_\ell}=T_{P_0}f_\ell-u_\ell$} \\
Input space                  & $H^s(\Omega)$ & $\cH_X\hookrightarrow H^{s+3/2}(\partial\Omega)$ & $H^s(\Omega)$ \\
Output space                 & $H^{s+2}(\Omega)\cap H_0^1(\Omega)$ & $\cH_Y\subset\cV_\cL$ & $H^{s+3}(\Omega)\cap H_0^1(\Omega)$ \\
Learned kernel               & $G(x,y)$ & $E_\cL(x,\theta)$ & $p(x,y)$ \\
Kernel domain                & $\Omega\subset\bbR^d$ & $\partial\Omega$ (dimension $d-1$) & $\Omega\subset\bbR^d$ \\
Output proj.\ rate (in $h_Y$)& $\cO(h_Y^{s+2-d/2})$ & $\cO(\eps_Y(h_Y))$ & $\cO(h_Y^{s+3-d/2})$ \\
Input proj.\ rate (in $h_X$) & $\cO(h_X^{s-d/2})$ & $\cO(h_X^{s+2-d/2})$ & $\cO(h_X^{s-d/2})$ \\ 
\hline
\end{tabular}%
}
\caption{Comparison of direct versus decomposed learning of solution operator of~\eqref{EQ:Diff} using matched RKHSs and the sampling assumptions of Corollary~\ref{COR:Global Rates}. The $m$-th order parametrix $P^{(m)}$ (Corollary~\ref{COR:Defect-m}) lifts the defect output rate to $\cO(h_Y^{s+2+(m+1)-d/2})$. Its volume-input rate and the harmonic column are unchanged, and convergence of the Levi series as $m\to\infty$ is not asserted.
}
\label{TAB:Rate Comparison}
\end{table}

For the Poisson problem, the defect vanishes. In the analytic regime, the stated upper bounds admit the following asymptotic comparison; this does not compare the actual errors or assert optimality of the direct baseline.
\begin{proposition}\label{PROP:Bound Comparison}
Assume $\gamma\equiv1$, $\sigma\equiv0$, and the analytic regime of Corollary~\ref{COR:Global Rates}, including its compatible spaces and sampling estimates, for the induced boundary-data class. Use ideal matrices $\bM^\star$ and assume the direct interpolation stability constants are uniformly bounded. Compare volume and boundary input grids at a common fill-distance scale $h_X$, and output grids at scale $h_Y$. Let $\cE_{\mathrm{dir}}$ and $\cE_{\mathrm{dec}}$ denote the right-hand sides of~\eqref{EQ:Total Error Bound} and~\eqref{EQ:Total Error Corrector Form}, with fixed positive upper-bound constants and $\eps_Y(h_Y)=\widetilde C_Y e^{-c_Y/h_Y}$. Then
\[
   {e^{-c_Y /h_Y}}/{h_Y^{\,s+2-d/2}}\xrightarrow[h_Y\to0]{}0,
   \qquad
   {h_X^{\,s+2-d/2}}/{h_X^{\,s-d/2}}=h_X^{2}\xrightarrow[h_X\to0]{}0 ,
\]
so there is $h_0>0$ with $\cE_{\mathrm{dec}}\le\cE_{\mathrm{dir}}$ for all $h_X,h_Y\le h_0$.
\end{proposition}
\begin{proof}
Let us write $\cE_{\mathrm{dir}}=A_Y+A_X$ and $\cE_{\mathrm{dec}}=B_Y+B_X$, where $A_Y=a_Yh_Y^{s+2-d/2}$, $A_X=a_Xh_X^{s-d/2}$, $B_Y=b_Ye^{-c_Y/h_Y}$, and $B_X=b_Xh_X^{s+2-d/2}$. The hypotheses allow fixed positive constants $a_Y,a_X,b_Y,b_X$, with $b_Y=M_{\cV}\widetilde C_Y$ and $b_X=L_E R_\partial\widetilde C_X$. Since $e^{-c_Y/h}=o(h^k)$ for every $k>0$, $B_Y/A_Y\to0$, while $B_X/A_X=(b_X/a_X)h_X^2\to0$. Thus $B_Y\le A_Y$ and $B_X\le A_X$ for sufficiently small fill distances, proving the comparison.
\end{proof} 

Let us finish this section with the following remark. The energy inner product~\eqref{EQ:Energy IP} is symmetric precisely because $\cL$ is self-adjoint. This is what makes $G$ symmetric and a reproducing kernel (Theorem~\ref{THM:HG Gen}). The construction therefore extends unchanged to a symmetric matrix coefficient, $\cL=-\nabla\cdot(A\nabla)+\sigma$ with $A=A^\top$ uniformly elliptic and $\sigma\ge0$, using the energy form $a(u,v)=\int_\Omega(A\nabla u\cdot\nabla v+\sigma uv)\,dx$. With a genuine drift term $b\cdot\nabla$, the operator is no longer self-adjoint. Therefore, $G$ loses its symmetry, and the orthogonal direct sum of Theorem~\ref{THM:Direct Sum Gen} fails. The decomposition $G=P_0-H_0$ and the learning reformulation~\eqref{EQ:wf BVP}, nonetheless, still make sense, but the analysis must then proceed in classical Sobolev spaces, with the adjoint Green's function playing the role of the second factor.

\section{Extension to nonlinear solution operators}
\label{SEC:Nonlinear}
 
We now generalize the kernel decomposition framework to learning a weakly nonlinear (semi-linear) solution operator. We show that the kernel decomposition continues to carry the singular part analytically, leaving only a smooth nonlinear correction. The main contrast to the linear case is that the coordinate map of the full solution operator is now nonlinear, so it is no longer represented exactly by a matrix. Rather than approximating it with a nonlinear function class, we exploit the decomposition iteratively: the only learned object remains a linear operator, fitted exactly as in Section~\ref{SEC:Learning-RKHS}. 
The nonlinearity is resolved by a contractive fixed-point iteration whose error bound is set by the linear surrogate alone.

\subsection{Nonlinear setup}

With the same assumptions on $\Omega$ and $\gamma$, let $u$ be the solution to the semilinear Dirichlet problem
\begin{equation}\label{EQ:Diff Nonl}
   \cL_0 u + q(u) = -\nabla\cdot(\gamma\nabla u) + q(u) = f\ \text{ in }\Omega,
   \qquad u=0\ \text{ on }\partial\Omega,
\end{equation}
with a nonlinearity $q\in \cC^\infty(\bbR)$ that is monotone and pins the trivial state, namely,
\begin{equation}\label{EQ:Monotone q}
   q(0)=0,\quad \mbox{and}\quad q'(t)\ge 0\quad\forall\,\, t\in\bbR.
\end{equation}
A canonical example of $q$ is the cubic $q(u)=u^3$. More generally, 
the choice $q(u)=\sum_{j\ge1}\beta_j\,u^{2j-1}$, with $\beta_j\ge0$, appears as the semilinear diffusion model in quantitative photoacoustics~\cite{KrLiReScZh-SIAM23,ReZh-SIAM26}.

We retain the source and solution spaces as before, 
$\cH_{\mathrm{src}}=H^s(\Omega)$ and $\cH_{\mathrm{sol}}=H^{s+2}(\Omega)\cap H_0^1(\Omega)$ with $s>d/2$, and write $T:f\mapsto u$ for the (now nonlinear) solution operator of~\eqref{EQ:Diff Nonl}. Throughout the section, $\tau$ denotes the torsion function of $\cL_0$, that is, the function of Lemma~\ref{LEM:Poisson Assumptions} with $\sigma\equiv0$ (the comparison bounds below require this larger torsion function).

\begin{proposition}\label{PROP:Well-posedness Nonl}
Under the assumption~\eqref{EQ:Monotone q}, ~\eqref{EQ:Diff Nonl} has a unique solution for every $f\in\cH_{\mathrm{src}}$. Moreover, we have
\begin{enumerate}[label=(\roman*)]
\item $T:\cH_{\mathrm{src}}\to\cH_{\mathrm{sol}}$ is bounded on bounded sets and $L^\infty$-Lipschitz with $L_T\le\|\tau\|_{L^\infty(\Omega)}$. Hence, $T$ satisfies assumptions~\hyperlink{Assum-A1}{\Ai}--\hyperlink{Assum-A2}{\Aii} on any bounded $\cM\subset\cH_{\mathrm{src}}$;
\item $T$ is Fr\'echet $\cC^\infty$, with $DT(f)=(\cL_0+q'(Tf))^{-1}$. If $q$ is real-analytic, so is $T$.
\end{enumerate}
\end{proposition}
\begin{proof}
Fix $M\ge\|\tau\|_{L^\infty}\|f\|_{L^\infty}$ and let $q_M(t):=q(\max\{-M,\min\{t,M\}\})$, a bounded, Lipschitz, nondecreasing function with $q_M=q$ on $[-M,M]$ and $q_M(0)=0$. Let $Q_M'=q_M$ with $Q_M(0)=0$, so that $Q_M\ge0$ is convex with at most quadratic growth. The functional $u\mapsto\int_\Omega\bigl(\tfrac{\gamma}{2}|\nabla u|^2+Q_M(u)-fu\bigr)dx$ is finite, strictly convex, coercive and continuous on $H_0^1(\Omega)$, hence has a unique minimizer $u_M\in H_0^1(\Omega)$, which is the unique weak solution of $\cL_0u_M+q_M(u_M)=f$~\cite{Evans-Book10}. The comparison argument below uses only that the nonlinearity is nondecreasing and Lipschitz on the range of the solutions (with $c$ the corresponding difference quotient), so it applies to $q_M$. With $f_2=0$, it gives $\|u_M\|_{L^\infty}\le\|\tau\|_{L^\infty}\|f\|_{L^\infty}\le M$. Hence $q_M(u_M)=q(u_M)$ and $u_M$ solves~\eqref{EQ:Diff Nonl}. Uniqueness for~\eqref{EQ:Diff Nonl} follows from the same comparison argument applied to two solutions, both of which are bounded by the same argument. No growth condition on $q$ is needed. Elliptic regularity for $\cL_0 u=f-q(u)$ and the Banach-algebra property of $H^s(\Omega)$ ($s>d/2$) give $u\in\cH_{\mathrm{sol}}$ with a bound on bounded sets. 

For the Lipschitz estimate, we set $v=Tf_1-Tf_2$. Subtracting the two equations and using the mean value theorem, $\cL_0 v + c\,v=f_1-f_2$ with $c(x)=q'(\xi(x))\ge0$ and $v|_{\partial\Omega}=0$. The Green's function $G_c$ of $\cL_0+c$ is nonnegative and, by the comparison principle, dominated by that of $\cL_0$, so \[\|v\|_{L^\infty}\le\bigl(\sup_x\!\int_\Omega G_c(x,y)\,dy\bigr)\|f_1-f_2\|_{L^\infty} \le\|\tau\|_{L^\infty}\|f_1-f_2\|_{L^\infty},\]
leading to $L_T\le\|\tau\|_{L^\infty}$. Finally, the map $\Psi(u,f):=\cL_0 u+q(u)-f$ has $\partial_u\Psi(u,f)=\cL_0+q'(u)$, which is boundedly invertible from $H^{s+2}(\Omega)\cap H_0^1(\Omega)$ onto $H^s(\Omega)$ (coercivity with $q'(u)\ge0$ gives invertibility at the $H_0^1$ level, and elliptic regularity lifts it to the stated spaces, which is the setting the implicit function theorem between $\cH_{\mathrm{src}}$ and $\cH_{\mathrm{sol}}$ requires). The implicit function theorem yields $T\in \cC^\infty$ (analytic if $q$ is), with the stated derivative.
\end{proof}

\subsection{Iterative decomposition}
\label{SUBSEC:Iterative Decomp}
 
We now develop an iterative procedure assuming that the nonlinear effect is weak. Fix a \emph{frozen coefficient} $\sigma_*\in \cC^\infty(\overline\Omega)$ with
$\sigma_*\ge0$ and set
\begin{equation}\label{EQ:Chord Operator}
   \cL_*:=\cL_0+\sigma_*=-\nabla\cdot(\gamma\nabla)+\sigma_*,
   \qquad S_*:=\cL_*^{-1}\ \text{ with zero Dirichlet data}.
\end{equation} 
We consider the chord (frozen-coefficient Newton) iteration for~\eqref{EQ:Diff Nonl}: 
$u^{(k+1)}=u^{(k)}-S_*\bigl(\cL_0u^{(k)}+q(u^{(k)})-f\bigr)$. Applying $S_*\cL_*=I$ on $H_0^1(\Omega)$ removes the differential operator from the right-hand side and leaves the implementable source form
\begin{equation}\label{EQ:Chord Source Form}
   u^{(k+1)}=S_*\bigl[\,f-q(u^{(k)})+\sigma_*\,u^{(k)}\,\bigr],
   \qquad k=0,1,2,\dots
\end{equation}
which requires only point values of $u^{(k)}$. The fixed point of~\eqref{EQ:Chord Source Form} is $u=S_*[f-q(u)+\sigma_*u]$, which is equivalent to $\cL_0u+q(u)=f$.

\begin{lemma}\label{LEM:Chord Contraction}
Let $J\subset\bbR$ be an interval containing the ranges of $u$ and of all iterates $\{u^{(k)}\}_{k=0}^\infty$, and set
\begin{equation}\label{EQ:Kappa}
\kappa_*:=\|\tau_{\sigma_*}\|_{L^\infty(\Omega)}\,
\sup_{x\in\overline\Omega,\,t\in J}\bigl|\sigma_*(x)-q'(t)\bigr| ,
\end{equation}
where $\tau_{\sigma_*}$ solves $\cL_*\tau_{\sigma_*}=1$ in $\Omega$, $\tau_{\sigma_*}|_{\partial\Omega}=0$. Then the exact iteration~\eqref{EQ:Chord Source Form} satisfies 
\[\|u^{(k+1)}-u\|_{L^\infty}\le\kappa_*\|u^{(k)}-u\|_{L^\infty}.\] Moreover, $\|\tau_{\sigma_*}\|_{L^\infty}\le\|\tau\|_{L^\infty}$ with $\tau$ the torsion function of Lemma~\ref{LEM:Poisson Assumptions}.
\end{lemma}
\begin{proof}
Let us write $e^{(k)}=u^{(k)}-u$. Subtracting $u=S_*[f-q(u)+\sigma_*u]$ from~\eqref{EQ:Chord Source Form} and using the mean value theorem, we have that $q(u^{(k)})-q(u)=q'(\xi)e^{(k)}$ pointwise for some $\xi$ between $u^{(k)}$ and $u$. Therefore, we have
\[
   e^{(k+1)}=S_*\bigl[(\sigma_*-q'(\xi))\,e^{(k)}\bigr].
\]
Since $\sigma_*\ge0$, the Green's function $G_*$ of $\cL_*$ is nonnegative. This gives us
\[
    \|S_*g\|_{L^\infty}\le\bigl(\sup_x\int_\Omega G_*(x,y)dy\bigr)\|g\|_{L^\infty}
=\|\tau_{\sigma_*}\|_{L^\infty}\|g\|_{L^\infty}\,.
\]
Applying this with
$g=(\sigma_*-q'(\xi))e^{(k)}$ gives the contraction. The comparison
$\|\tau_{\sigma_*}\|_\infty\le\|\tau\|_\infty$ follows from $G_*\le G_0$. The latter is a consequence of the maximum principle and $\sigma_*\ge0$.
\end{proof}

Let $\wh S_*$ denote the decomposed surrogate for $S_*$ built in Section~\ref{SUBSEC:Improved Algo}, and run
\begin{equation}\label{EQ:Chord Perturbed}
   \wh u^{(k+1)}=\wh S_*\bigl[\,g^{(k)}\,\bigr],
   \qquad g^{(k)}:=f-q(\wh u^{(k)})+\sigma_*\,\wh u^{(k)} .
\end{equation}
We write $\psi(x,t):=q(t)-\sigma_*(x)\,t$. The source reads $g^{(k)}=f-\psi(\cdot,\wh u^{(k)})$. We assume the natural invariance hypothesis, which holds whenever $\kappa_*<1$ and the surrogate is accurate enough:

\smallskip
\noindent ($\bH$)\ \ There is a bounded set $\cM_*\subset\cH_{\mathrm{src}}$ with
$g^{(k)}\in\cM_*$ for all $k$, and the surrogate satisfies the uniform bound
\begin{equation}\label{EQ:Eta}
   \eta_*:=\sup_{g\in\cM_*}\bigl\|S_*g-\wh S_*g\bigr\|_{L^\infty(\Omega)}<\infty .
\end{equation}
\smallskip
Hypothesis ($\bH$) is an invariance statement. While we pose it as a hypothesis, in Appendix~\ref{SEC:Proof of Hypothesis  (H)}, we provide a short proof that this hypothesis actually holds for small data $f$.

Under hypothesis ($\bH$), the convergence of the learned chord iteration follows from the standard argument.
\begin{theorem}\label{THM:Chord Convergence}
Assume $\kappa_*<1$, that the ranges of $Tf$ and of all perturbed iterates $\{\wh u^{(k)}\}_{k\ge0}$ lie in the interval $J$ of~\eqref{EQ:Kappa}, and {\rm ($\bH$)}. Then for every $k\ge0$,
\begin{equation}\label{EQ:Chord Error}
   \bigl\|\wh u^{(k)}-Tf\bigr\|_{L^\infty(\Omega)}
   \;\le\;\kappa_*^{\,k}\,\bigl\|\wh u^{(0)}-Tf\bigr\|_{L^\infty(\Omega)}
   \;+\;\frac{1-\kappa_*^{\,k}}{1-\kappa_*}\;\eta_* ,
\end{equation}
and consequently
\begin{equation}\label{EQ:Chord Limsup}
   \limsup_{k\to\infty}\bigl\|\wh u^{(k)}-Tf\bigr\|_{L^\infty(\Omega)}
   \;\le\;\frac{\eta_*}{1-\kappa_*} .
\end{equation}
\end{theorem}
\begin{proof}
Let $e^{(k)}=\wh u^{(k)}-u$ with $u=Tf=S_*[f-q(u)+\sigma_*u]$. Splitting
\[
   \wh u^{(k+1)}-u=\underbrace{\bigl(\wh S_*-S_*\bigr)g^{(k)}}_{\text{surrogate error}}
   \;+\;\underbrace{S_*\bigl[g^{(k)}-\bigl(f-q(u)+\sigma_*u\bigr)\bigr]}_{\text{contraction}},
\]
the first term is bounded by $\eta_*$ using ($\bH$), and the second equals $S_*[(\sigma_*-q'(\xi))e^{(k)}]$ exactly as in Lemma~\ref{LEM:Chord Contraction} (Here, the mean-value point $\xi$ lies between $\wh u^{(k)}(x)$ and $u(x)$, which is why the ranges of the perturbed iterates and of $u=Tf$ are assumed to lie in the interval $J$ defining $\kappa_*$. Both are guaranteed under the small-data hypothesis of Appendix~\ref{SEC:Proof of Hypothesis  (H)}. The latter follows because the map $v\mapsto S_*[f-\psi(\cdot,v)]$ is a $\kappa_*$-contraction of the $L^\infty$-ball of radius $\rho$ into itself, so its fixed point $u$ lies in that ball). Hence it is bounded by $\kappa_*\|e^{(k)}\|_{L^\infty}$. Therefore, $\|e^{(k+1)}\|_{L^\infty}\le\eta_*+\kappa_*\|e^{(k)}\|_{L^\infty}$, and iterating this scalar recursion gives the geometric sum~\eqref{EQ:Chord Error}. Letting $k\to\infty$ with $\kappa_*<1$ yields~\eqref{EQ:Chord Limsup}.
\end{proof}

The inequality~\eqref{EQ:Chord Limsup} bounds the nonlinear error by the linear approximation error $\eta_*$, amplified by $(1-\kappa_*)^{-1}$. The linear rates therefore transfer whenever their hypotheses hold uniformly for all transformed inputs generated by the iteration, including the defect-potential contribution.
\begin{corollary}\label{COR:Chord Rates}
Let $\wh S_*$ be the decomposed approximation of Section~\ref{SUBSEC:Improved Algo} for $\cL_*$. Assume the hypotheses of Corollary~\ref{COR:Global Rates} hold uniformly for the boundary traces $h_g=(T_{P_0}g)|_{\partial\Omega}$, $g\in\cM_*$, with $\|h_g\|_{\cH_X}\le R_\partial$, and use the ideal projected matrix for the harmonic extension. Let $D_*g=S_*\mathcal F_{0,*}g$ denote the defect-potential operator, where $\mathcal F_{0,*}=\cL_*T_{P_0}-I$, and let $\wh D_*$ be its approximation. Set
\[
\Delta_{\mathrm{def}}:=\sup_{g\in\cM_*}
\|D_*g-\wh D_*g\|_{L^\infty(\Omega)}.
\]
For the componentwise reconstruction $\wh S_*=T_{P_0}-\wh E_{\cL_*}\mathfrak R T_{P_0}-\wh D_*$,~\eqref{EQ:Total Error Corrector Form} and the triangle inequality give
\[
   \eta_*\;\le\;M_{\cV}\,\eps_Y(h_Y)
   +L_E R_\partial\,\widetilde C_X\,h_X^{\,s+2-d/2}
   +\Delta_{\mathrm{def}},
\]
where $M_{\cV}=R_\partial\|E_{\cL_*}\|_{\cH_X\to\cH_Y}$. The Sobolev regime of Corollary~\ref{COR:Global Rates}(i) gives $\eps_Y(h_Y)\le C_Yh_Y^{t-d/2}$. The exponential regime (ii) additionally requires analytic $\partial\Omega$, $\gamma$, and $\sigma_*$, together with the uniform analytic input, output and sampling assumptions for the entire class $\{h_g:g\in\cM_*\}$. 
Moreover, under the hypotheses of Theorem~\ref{THM:Chord Convergence}, with the solution and all perturbed iterates taking values in the interval $J$ of~\eqref{EQ:Kappa}, we have
\begin{equation}\label{EQ:Chord Total}
   \limsup_{k\to\infty}\bigl\|\wh u^{(k)}-Tf\bigr\|_{L^\infty(\Omega)}
   \;\le\;\frac{1}{1-\kappa_*}
   \Bigl[\,M_{\cV}\eps_Y(h_Y)
   +L_E R_\partial\,\widetilde C_X\,h_X^{\,s+2-d/2}
   +\Delta_{\mathrm{def}}\,\Bigr].
\end{equation}
\end{corollary} 

Table~\ref{TAB:Chord Comparison} summarizes the conditional rates for the semilinear problem. Only linear corrections are learned: their finite-dimensional projections, not the full inverse $S_*$, admit exact matrix representations. The iteration adds a contraction requirement and logarithmically many applications per query.
\begin{table}[!htb]
\centering
\renewcommand{\arraystretch}{1.4}
\begin{tabular}{lc}
\hline
 & \textbf{Iterative decomposition for nonlinear problem} \\
\hline
Learned object            & Projected linear corrections (ideal matrices) \\
Harmonic output rate      & $\cO(\eps_Y(h_Y))$ \\
Boundary input rate       & $\cO(h_X^{\,s+2-d/2})$ \\
Defect contribution       & $\Delta_{\mathrm{def}}$ \\
Cost per query            & $\cO(\log(1/\eta_*))$ applications of $\wh S_*$ \\
Limitation               & $\kappa_*<1$ (small data; see Appendix~\ref{SEC:Proof of Hypothesis  (H)}) \\
\hline
\end{tabular}
\caption{Rates for~\eqref{EQ:Diff Nonl} under Corollary~\ref{COR:Chord Rates}. The linear approximation hypotheses must hold uniformly along the iteration. The total limiting error includes the defect contribution and is amplified by $(1-\kappa_*)^{-1}$.}
\label{TAB:Chord Comparison}
\end{table}

By~\eqref{EQ:Chord Error}, the transient $\kappa_*^k\|\wh u^{(0)}-Tf\|_\infty$ falls
below the floor $\eta_*/(1-\kappa_*)$ once
\[
   k\;\ge\;\frac{\log\bigl(\|\wh u^{(0)}-Tf\|_{L^\infty}(1-\kappa_*)/\eta_*\bigr)}
                {\log(1/\kappa_*)}
   \;=\;\cO\!\left(\frac{\log(1/\eta_*)}{\log(1/\kappa_*)}\right).
\]
When the algebraic input term dominates $\eta_*$ this is $\cO(\log(1/h_X))$, so only
\emph{logarithmically many} applications of $\wh S_*$ are needed. Each learned correction is evaluated by an encode--matrix--decode product, together with the computable potential in the reconstructed solve. Relative to one such linear solve, the iteration adds only a logarithmic factor.

Let us end this section with the following remark. Full Newton replaces $\sigma_*$ by $q'(u^{(k)})$ at every step and converges quadratically, so the iteration count improves to $\cO(\log\log(1/\eta_*))$ iterations. The linearized operator is $DT(f)^{-1}=\cL_0+q'(Tf)$, again of the form treated in Section~\ref{SEC:Elliptic}, so each individual solve remains decomposable. As a \emph{learning} scheme, however, the operator changes from step to step, so one must learn the coefficient-parametric family  $(\sigma,g)\mapsto(\cL_0+\sigma)^{-1}g$, which is nonlinear in $\sigma$ and therefore falls outside the exactly-representable linear framework of this paper, reinstating a genuine approximation error. Since the gain is only $\log(1/\eta_*)\to\log\log(1/\eta_*)$ in the iteration count while the loss is the entire exact-representation advantage, freezing the coefficient is the better trade, and we adopt it. Learned Newton steps and Picard-type iterations built on neural operators have been studied in~\cite{HaLiYa-arXiv24,TaSo-arXiv26}. The scheme above differs in that the iterated object is an exactly representable linear surrogate, whose approximation rates the iteration inherits.

\section{Numerical experiments}
\label{SEC:Numer}

We now present numerical experiments illustrating the approximation gains of the kernel decomposition strategy.

\subsection{Setup for simulations} 

We consider the problem on the unit disk $\Omega = \{x \in \bbR^2 : |x| < 1\}$, which has an analytic boundary. For illustration, we discretize $\Omega$ on a Cartesian grid of $N_{1d}\!\times\!N_{1d}$ 
points masked to the disk, with $N_\Omega\approx \pi N_{1d}^2/4$ interior points, mesh spacing $h_{\mathrm g}=2/(N_{1d}-1)$ (not to be confused with the fill distances $h_X,h_Y$), and quadrature weight $h_{\mathrm g}^2$ per cell. The boundary $\partial\Omega$ is sampled at $m_b$ uniformly spaced angles $\theta_k=2\pi k/m_b$ with arc-length weight $2\pi/m_b$, so that trapezoidal integration of periodic boundary data is spectrally accurate.

We compare the direct learning approach with our kernel decomposition approach. Both methods parameterize a separable kernel 
\[
\wh K(x,\xi)=\Psi_{\mathrm{out}}(x)^\top C\,\Psi_{\mathrm{in}}(\xi),
\]
with $\Psi_{\mathrm{out}}$ a tensor-product Chebyshev basis on $\Omega$ of order $p_{\mathrm{out}}$ in each spatial direction, hence $p_{\mathrm{out}}^2$ volume modes. They differ in the input-side basis: tensor-product Chebyshev on $\Omega$ of order $p_{\mathrm{in}}$ ($p_{\mathrm{in}}^2$ volume modes) for the direct method, and Fourier truncation of order $q_{\mathrm{bdy}}$ on $\partial\Omega$ ($2q_{\mathrm{bdy}}+1$ boundary modes) for the decomposed method. These basis choices are dictated by the input domains, not tuned for performance: tensor-product Chebyshev is the standard spectral basis on the bounded, non-periodic disk, whereas on the boundary circle $S^1$ the Fourier basis is the natural choice, being the eigenbasis of the boundary Laplacian and exact for smooth periodic data. Using the same basis for both would unfairly penalize one of them. The change of basis instead reflects the dimensional reduction of Section~\ref{SUBSEC:Improved Algo}. 

Training and test inputs are Gaussian source bumps $f_\ell(x)=\exp(-\beta\,|x-c_\ell|^2)$ with bandwidth $\beta=40$ and random centers $c_\ell$ drawn uniformly from a disk of radius $0.6$, so the centers remain separated from $\partial\Omega$; the Gaussian sources are not compactly supported inside $\Omega$. With fixed bandwidth and bounded centers, this family is uniformly bounded in each finite-order Sobolev norm and has uniform analytic bounds on fixed neighborhoods of $\overline\Omega$. Both methods are trained on the same $N=300$ source-solution pairs $\{(f_\ell,u_\ell)\}$: for each training source $f_\ell$, we generate the response $u_\ell = T(f_\ell)$ by applying the discrete Green's matrix to $f_\ell$. Identical quadrature is used for training and test inputs, with any quadrature error shared by both methods. The sample size is held fixed while the discretization parameters are swept to examine approximation effects. The condition $N\ge m_X$ alone does not establish negligible fitting error: recovery also depends on the encoded-data rank and conditioning, as discussed in Proposition~\ref{PROP:Linear Identifiability}.

The coefficient matrix $C$ is obtained in a closed form by Tikhonov-regularized Galerkin least squares. The regularization parameter is fixed at $\lambda=10^{-10}$ throughout. Outside the convergence studies described below, we set $p_{\mathrm{in}}=q_{\mathrm{bdy}}=6$ and $p_{\mathrm{out}}=12$.

The bounds in Sections~\ref{SEC:Learning-RKHS} and~\ref{SEC:Elliptic} concern kernel-translate subspaces, whereas these experiments use Chebyshev and Fourier bases. They illustrate the role of regularity and boundary encoding, but do not directly verify the kernel power-function estimates: changing the approximation subspace need not preserve its convergence rate. In particular, Mercer's theorem does not imply equal rates for arbitrary subspaces of the same dimension. The Gaussian data are analytic, so an algebraic fit over the tested range should not be read as the optimal asymptotic rate of direct learning.

\subsection{Direct vs decomposition} 

We start with the case of $\gamma=1$ and $\sigma=0$, that is, learning the solution operator of the Dirichlet Poisson problem. On the unit disk, this problem admits a closed-form Green's function and decomposition
\begin{equation}\label{EQ:Poisson Kernel Disk}
\begin{aligned}
    \Phi(x,y) &= -\frac{1}{2\pi}\log|x-y|, \qquad
    G(x,y) = \Phi(x,y) - H(x,y), \\
    H(x,y)    &= -\frac{1}{4\pi}\log\bigl(1 - 2\,x\!\cdot\!y + |x|^2|y|^2\bigr).
\end{aligned}
\end{equation}

\begin{figure}[!htb]
  \centering
  \includegraphics[width=\linewidth]{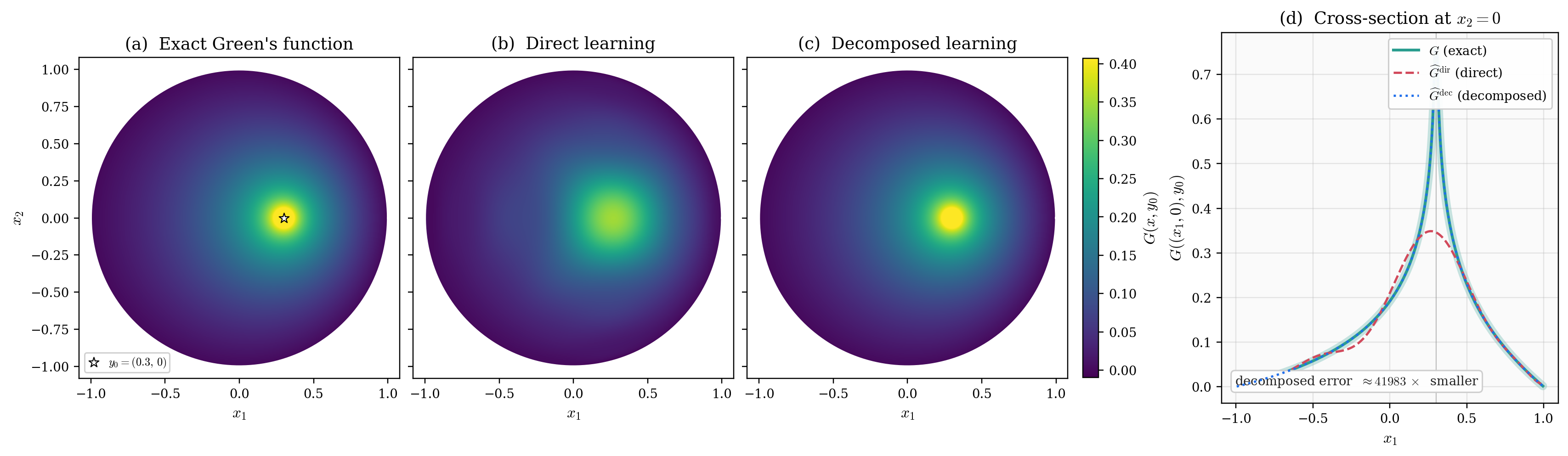}
  \caption{Learned Dirichlet Green's function on the unit disk at a fixed source point $y_0=(0.3,0)$.  (a)-(c) are, respectively, the analytical $G(x,y_0)$, $\wh G^{\mathrm{dir}}(x,y_0)$ (direct learning result), and $\wh G^{\mathrm{dec}}(x,y_0)$ (decomposition learning result). (d) overlays the three kernels along the horizontal line $\{x_2=0\}$ through $y_0$. 
  }
    \label{FIG:Visualization}
\end{figure}
In Figure~\ref{FIG:Visualization}, the kernels recovered by these different methods are visualized. Taking a fixed source point $y_0 \in \Omega$, we plot the column $\wh G(\,\cdot\,, y_0)$ of the learned kernel for both methods alongside the analytical form of the ground truth $G(\,\cdot\,, y_0)$ from~\eqref{EQ:Poisson Kernel Disk} on a common color scale (Figure~\ref{FIG:Visualization}). The smooth Chebyshev basis adopted by the direct method fails to reproduce the logarithmic peak of $G$ at $y_0$, leaving a visibly flattened approximation in a neighborhood of the source point. The method based on the decomposition, by contrast, reproduces the local structure of $G$ to high accuracy, via the implicit reconstruction $\wh G^{\mathrm{dec}} = \Phi - \int_{\partial\Omega}\wh E(\cdot,\theta)\,\Phi(\theta,y_0)\,dS(\theta)$, in which the singular part is supplied analytically by $\Phi$ and only the smooth corrector is learned.

\begin{figure}[!htb]
  \centering
  \includegraphics[width=\linewidth]{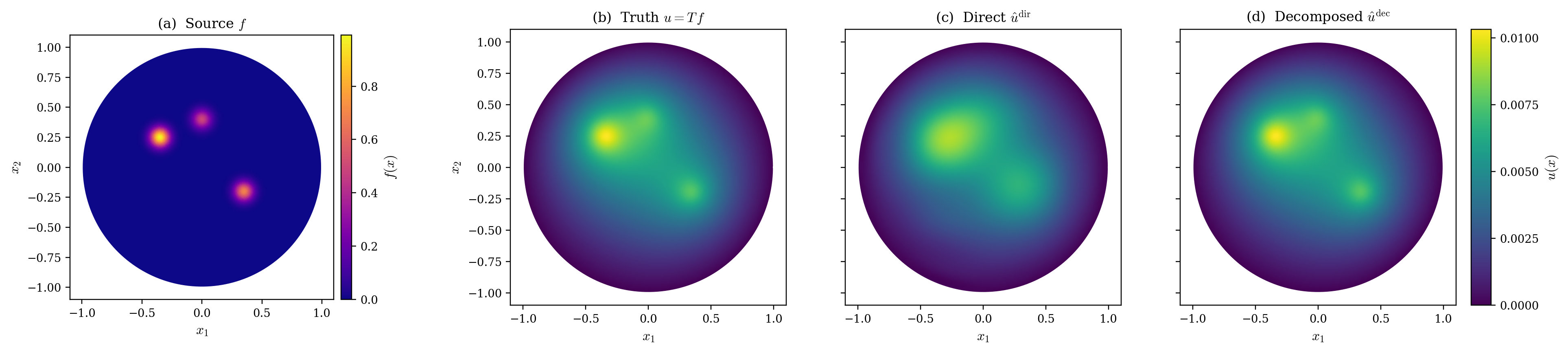}
 \caption{Solution-operator comparison on an out-of-training source of three sharp positive Gaussian bumps, $f(x)=\exp(-\beta_{\mathrm{src}}|x-c_1|^2) +0.7\,\exp(-\beta_{\mathrm{src}}|x-c_2|^2)+0.5\,\exp(-\beta_{\mathrm{src}}|x-c_3|^2)$, with $c_1=(-0.35,0.25)$, $c_2=(0.35,-0.20)$, $c_3=(0,0.40)$ and bandwidth $\beta_{\mathrm{src}}=160$. Panel~(a) shows the source $f$. Panels~(b)--(d) show the truth $u=T f$, the direct prediction $\wh u^{\mathrm{dir}}$, and the decomposed prediction $\wh u^{\mathrm{dec}}$ on a common color scale. 
 }
    \label{FIG:Const Operator Three Point}%
\end{figure}
Next, we compare the two learning strategies at the solution-operator level, applying each learned operator to test sources outside the training family and plotting $\wh u = \wh T(f)$ against the ground truth $u = T(f)$. In Figure~\ref{FIG:Const Operator Three Point}, the source is three sharp positive Gaussian bumps: the direct prediction $\wh u^{\mathrm{dir}}$ blurs the sharp features together, while the decomposition-based prediction $\wh u^{\mathrm{dec}}$ keeps them distinct. In Figure~\ref{FIG:Const Operator Oscillatory} the source is the globally oscillatory $f(x)=\sin(8\pi x_1)\sin(8\pi x_2)$ (panel~(a)), whose response $u=Tf$ retains the oscillatory pattern at reduced amplitude (panel~(b)). The direct estimate $\wh u^{\mathrm{dir}}$ loses the oscillation entirely, replacing it with a few large smooth regions and visible errors near the boundary (panel~(c)), whereas the decomposed estimate $\wh u^{\mathrm{dec}}$ reproduces the checkerboard pattern and stays visibly close to $u$ across the domain (panel~(d)).
\begin{figure}[!htb]
 \centering
 \includegraphics[width=\linewidth]{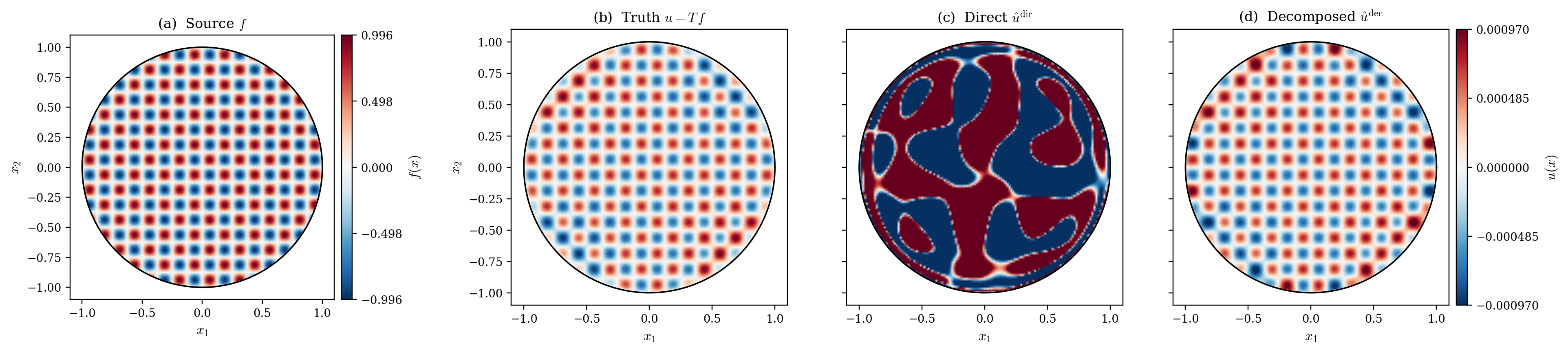}
\caption{Solution-operator comparison on the out-of-training oscillatory source $f(x)=\sin(8\pi x_1)\sin(8\pi x_2)$. (a): source $f$. (b)--(d): the truth $u=Tf$, the direct prediction $\wh u^{\mathrm{dir}}$, and the decomposed prediction $\wh u^{\mathrm{dec}}$ on a common symmetric scale $9.7\times10^{-4}$.
}
 \label{FIG:Const Operator Oscillatory}%
\end{figure}

\subsection{Input and output rates}

To examine the input-side convergence, we hold $p_{\mathrm{out}}{=}12$ fixed and fit the range before the output and computational error floors become visible. We sweep the input-side basis size for both methods and show the mean $L^\infty$ test error $\cE$ against the input degrees of freedom $N_{\mathrm{in}}$ in the left plot of Figure~\ref{FIG:Input-Output Conv}. Here $N_{\mathrm{in}}$ counts the active input-side basis modes: it equals $p_{\mathrm{in}}^2$ volume Chebyshev modes for the direct method and $2q_{\mathrm{bdy}}+1$ boundary Fourier modes for the decomposition-based method. We summarize the power-law portion by a finite-range fitted index, not by a maximal Sobolev regularity of the sources.
\begin{figure}[!htb]
  \centering  \includegraphics[width=0.475\linewidth]{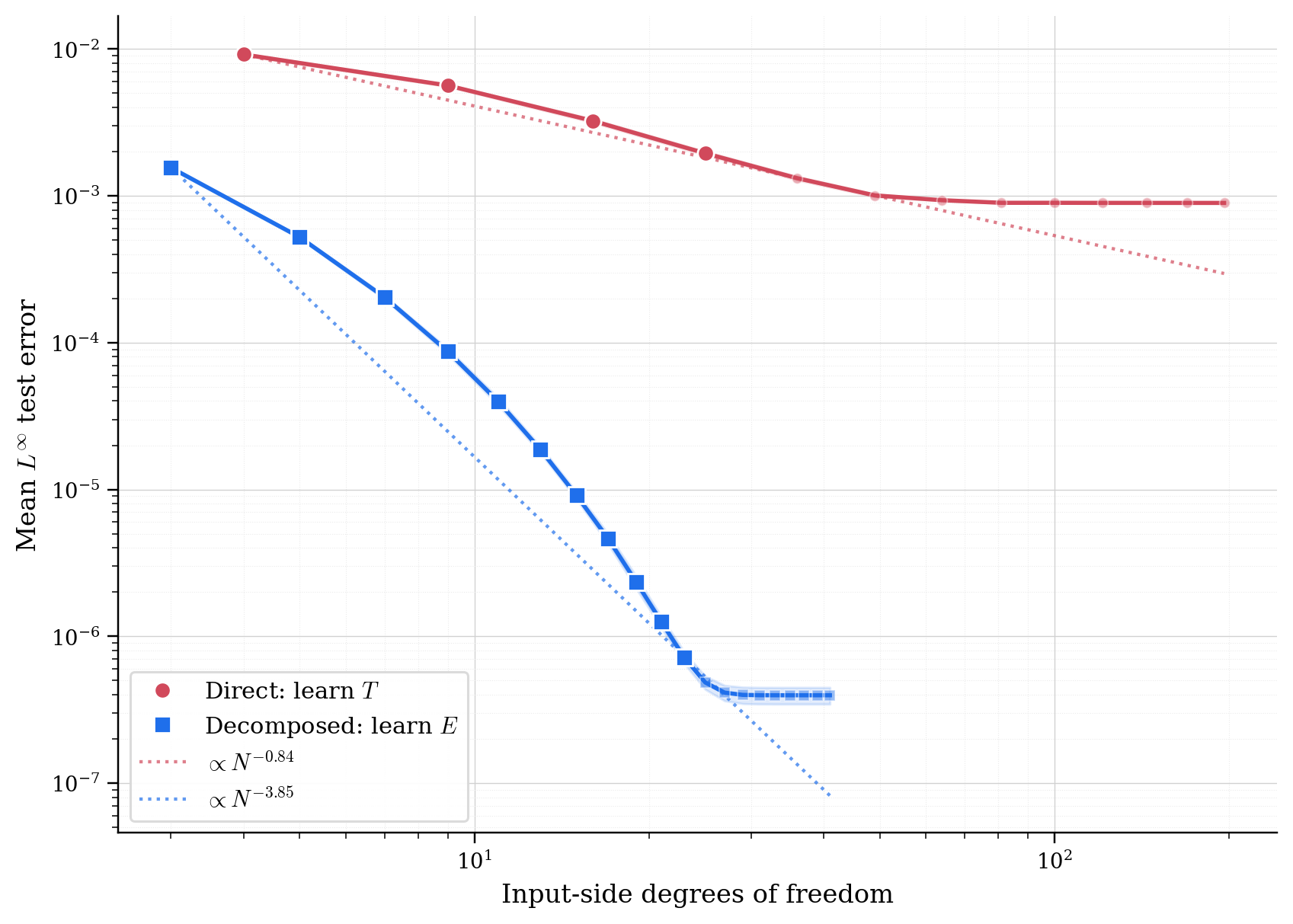} 
\includegraphics[width=0.475\linewidth]{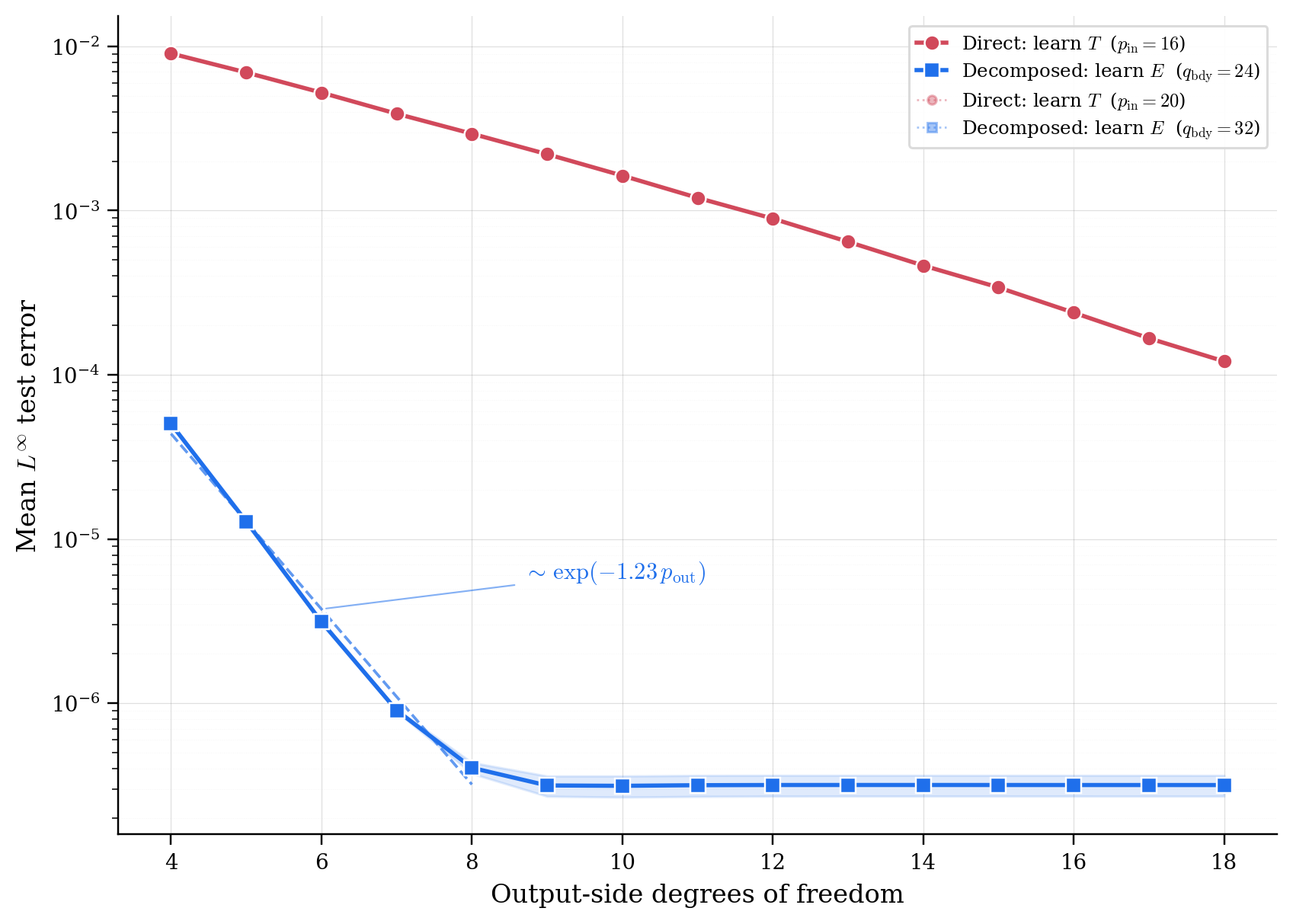}
  \caption{Input-side (left, log-log axes, $p_{\rm out}=12$) and output-side (right, semilog-$y$ axes, $p_{\rm in}=16$ and $q_{\rm bdy}=24$) convergence of $L^\infty$ error.
  }
\label{FIG:Input-Output Conv}
\end{figure}

Past their fitted ranges, the last several points of each curve leave the fitted slope (faded markers in the left plot of Figure~\ref{FIG:Input-Output Conv}) and are excluded from the fits. The direct curve levels off at around $8.9\times10^{-4}$: with the output basis frozen at $p_{\mathrm{out}}{=}12$, enriching the input cannot push the error below what that fixed output space can represent.  The decomposed curve levels off far lower, at around $4.0\times10^{-7}$, against a floor set jointly by the Tikhonov regularization $\lambda{=}10^{-10}$ and the finite grid and boundary-quadrature resolution. Both plateaus are artifacts of the discretization rather than of the approximation rates, and we return to them later in this section.
 
For comparison, substituting $h_X\asymp N_{\mathrm{in}}^{-1/d'}$, with $d'=2$ for volume samples and $d'=1$ for boundary samples, into the Sobolev input bounds gives the reference powers
\begin{equation}
  \underbrace{h_X^{\,s-d/2}\asymp N_{\mathrm{in}}^{-(s-1)/2}}_{\text{direct, }d'=d=2},
  \qquad
  \underbrace{h_X^{\,s+2-d/2}\asymp N_{\mathrm{in}}^{-(s+1)}}_{\text{decomposed, }d'=d-1=1}.
\end{equation} 
 
To examine the output-side convergence, we freeze the input bases at high resolution ($p_{\mathrm{in}}{=}16$ for the direct method, $q_{\mathrm{bdy}}{=}24$ for the decomposed one) so that the input-side error is negligible, and we sweep the shared output Chebyshev order $p_{\mathrm{out}}$ to isolate the output projection error in the right plot of Figure~\ref{FIG:Input-Output Conv}.  The
plot uses semilog-$y$ axes, on which an exponential decay in $p_{\mathrm{out}}$ appears as a straight line, and a slower, algebraic decay appears concave.  Read this way, the two methods separate sharply. The decomposed curve falls along a straight line through its pre-floor regime $p_{\mathrm{out}}\in\{4,\dots,8\}$, with least-squares decay constant around $e^{-1.23\,p_{\mathrm{out}}}$, and reaches a floor of $3\times10^{-7}$ by $p_{\mathrm{out}}{=}9$. Beyond this point, the approximation error is buried under the computational and regularization error, so further output modes bring no improvement. The direct curve decays much more slowly over the tested range, staying more than two orders of magnitude above the decomposed floor at $p_{\mathrm{out}}{=}18$.

\subsection{Variable-coefficient case}
\label{SUBSEC:Num Var Coeff}

We next present numerical results for the variable-coefficient case. In this case, we do not have access to the analytical formula of $\Phi^\cL$, so we rely on the computable decomposition strategy of Section~\ref{SEC:Computable-Decom}. We take $\cL$ with $\gamma=1+0.5\sin(\pi x_1)\cos(\pi x_2)$ and $\sigma\equiv 0$ on the unit disk, and write $G_\cL$ for the corresponding Green's function. Ground-truth responses are produced by a polar spectral collocation solver (Fourier in $\theta$, Chebyshev in $r$, and validated to $5\times10^{-12}$ on manufactured solutions). The parametrix potentials and their boundary traces at all orders are evaluated by fixed discrete quadratures shared between training and testing. Because the truth is now produced by an independent solver, the shared-quadrature cancellation of the constant coefficient case is only approximate, and the decomposition-based method inherits a consistency floor of order $10^{-5}$. Other parameters are kept the same as those in the constant-coefficient case (for instance, $N=300$).

\begin{figure}[!htb]
  \centering
  \includegraphics[width=\linewidth]{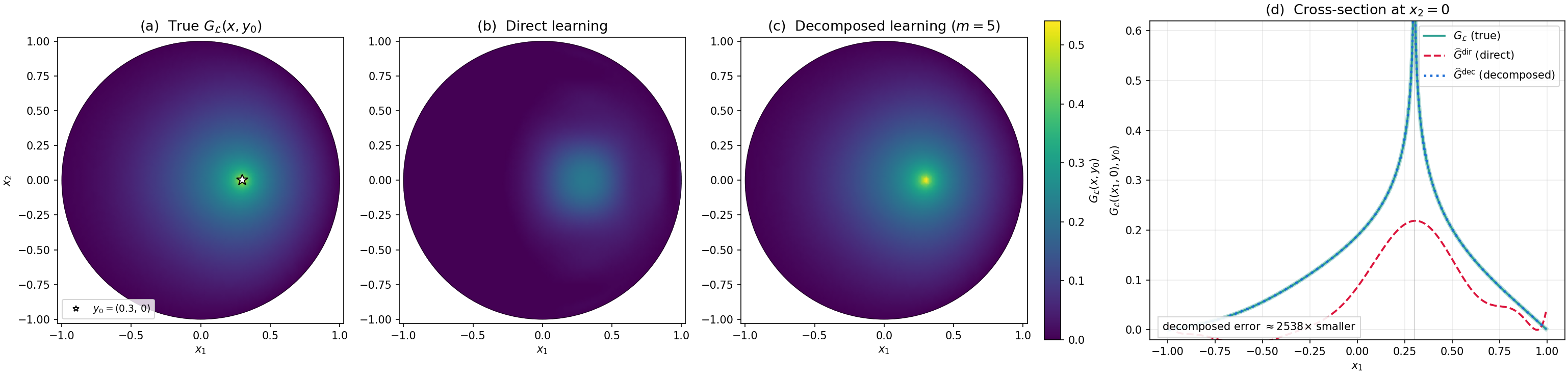}
  \caption{The analog of Figure~\ref{FIG:Visualization} for the variable-coefficient operator, with the order-$5$ Levi parametrix in place of $\Phi^\cL$. (a) reference column $G_\cL(x,y_0)$ from the
  parametrix-plus-remainder construction; (b) direct learning, evaluated on
  the unit-mass Gaussian surrogate of $\delta_{y_0}$ that its volume input
  basis requires; (c) decomposed learning, evaluated at the exact point
  source; (d) the three kernels along the line $\{x_2=0\}$ through $y_0$.}
  \label{FIG:VarCoef Kernel}
\end{figure}
Figure~\ref{FIG:VarCoef Kernel} shows the true and learned Green's function $G_\cL(\cdot,y_0)$. The ``true" Green's function is constructed independently of the learning pipeline, as $P^{(5)}(\cdot,y_0)-H_0^{(5)}$ with $\cL H_0^{(5)}=-\cF_0^{6}\delta_{y_0}$ and $H_0^{(5)}|_{\partial\Omega}=P^{(5)}(\cdot,y_0)$ (the order-$5$ analogue of $H_0=P_0-G$), solved spectrally after a harmonic lifting. Carrying out the same construction at order $m=1$ instead changes the column by at most $6\times10^{-6}$, so the reference is reliable at the $10^{-5}$ level of the experiment.

Figure~\ref{FIG:VarCoef Rates} shows the corresponding finite-range convergence fits in the same setting as Figure~\ref{FIG:Input-Output Conv}. On the input side, freezing $p_{\mathrm{out}}=12$ gives $N_{\mathrm{in}}^{-0.81}$ for the direct method against $N_{\mathrm{in}}^{-2.77}$ for the decomposition method. On the output side, freezing input bases gives a power-law fit $p_{\mathrm{out}}^{-2.13}$ for the direct method and an exponential fit $e^{-0.59\,p_{\mathrm{out}}}$ for the decomposition method. These fits describe the tested analytic data and resolution range, not an exponential guarantee for the full finite-order parametrix correction over general Sobolev source classes.
\begin{figure}[!htb]
  \centering
  \includegraphics[width=0.495\linewidth]{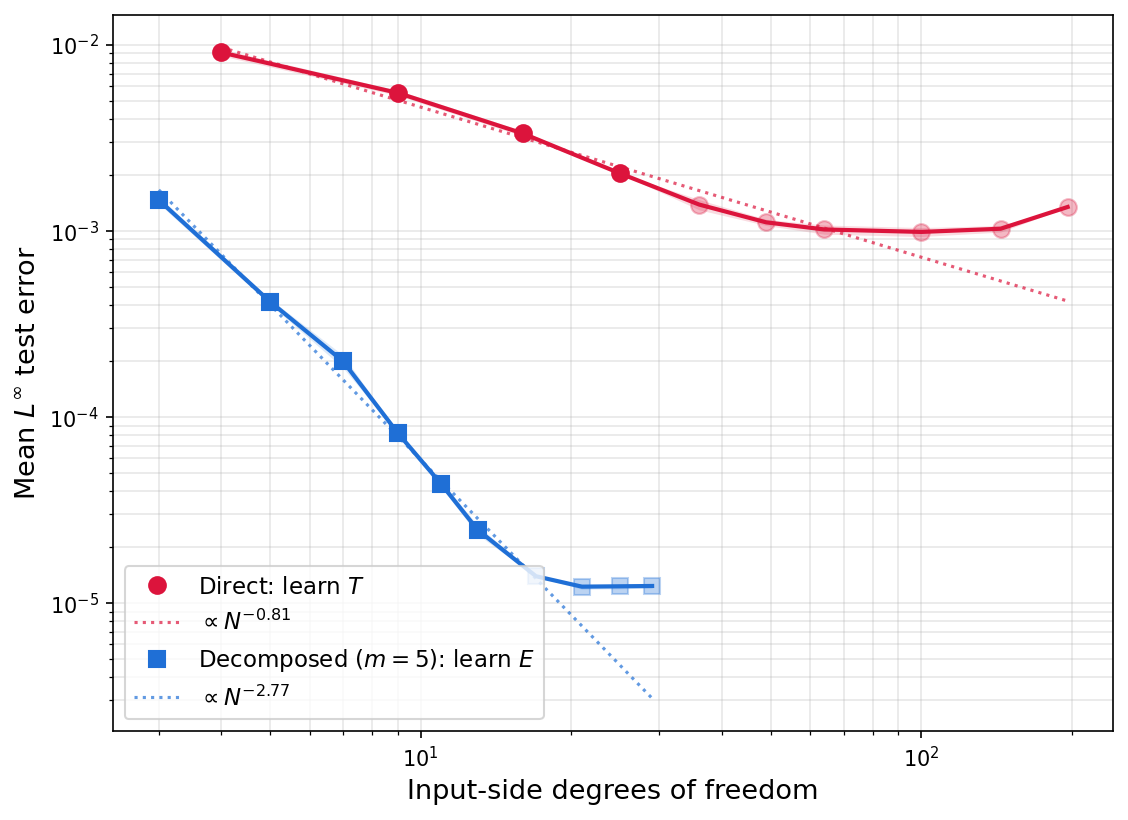}\hfill
  \includegraphics[width=0.495\linewidth]{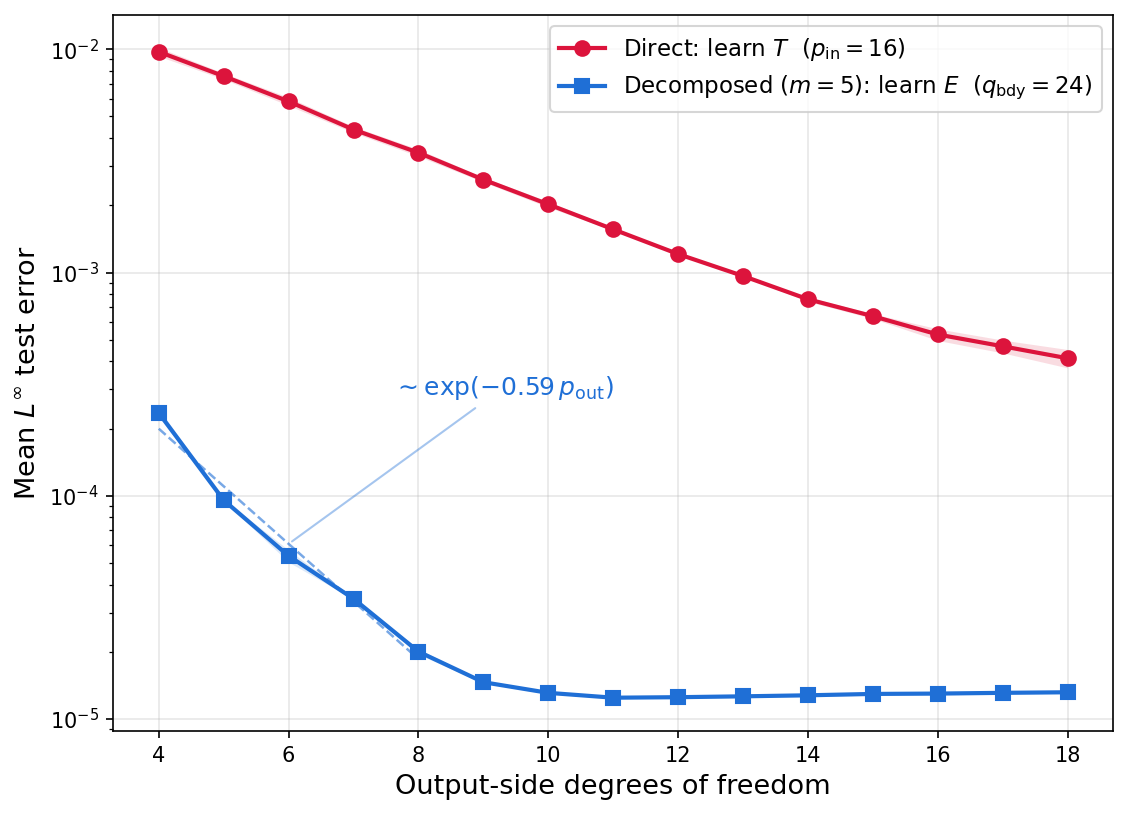}
  \caption{The analogs of Figure~\ref{FIG:Input-Output Conv} for the variable-coefficient case with $m=5$. Left: log-log axes with $p_{\mathrm{out}}=12$. Right: semilog-$y$ axes with $p_{\mathrm{in}}=16$ and $q_{\mathrm{bdy}}=24$.}
\label{FIG:VarCoef Rates}
\end{figure}

\section{Conclusion}
\label{SEC:Concl}

This work studies the problem of learning the solution operator, equivalently the Green's function, of self-adjoint second-order elliptic PDEs. We established the energy-space orthogonal decomposition $\cH_N\cong\cH_G\oplus_a\cH_{H^\cL_a}$ (for $\sigma\not\equiv0$) and turned the pointwise splitting $G=\Phi^\cL-H^\cL$ into a kernel-decomposition learning strategy. Using tools from reproducing kernel Hilbert space learning theory, we show the approximation advantage of the kernel decomposition learning strategy. For the exact decomposition, replacing volume sources by boundary traces combines a $3/2$-order Sobolev lift with the lower boundary dimension, improving the input power-function exponent by two. The global output bound is algebraic for general Sobolev data and improves with the regularity of the induced boundary data. An exponential bound requires the uniform analytic and approximation assumptions of Corollary~\ref{COR:Global Rates}. A finite-order parametrix also introduces a defect potential with its own regularity estimate. Preliminary numerical experiments illustrate these approximation gains.

There are a few interesting directions of future research. First, the most immediate extension is to consider the statistical effect. The bounds presented in this work are \emph{a priori} \emph{approximation} estimates and, apart from the exact-recovery condition of Proposition~\ref{PROP:Linear Identifiability}, do not involve the sample size $N$. Combining them with estimation bounds for the fitting
step would yield a complete sample-complexity theory and, in particular, the optimal balance $m_X=m_X(N)$ between projection accuracy and data requirement. Second, the global exponential output bound requires analytic geometry and coefficients, uniformly analytic boundary data, and compatible approximation spaces. On domains with corners, both the output regularity and the boundary-trace regularity must be reconsidered. Neither the exponential output rate nor the full input gain is automatic. A careful analysis of these domains is an important project. Third, our numerics are linear, noise-free, and confined to a single geometry. Extending the empirical study to noisy data through the Tikhonov-regularized encoders, to more general geometries, and to nonlinear operators are natural steps to take further. Finally, extending the framework to accommodate Robin and mixed boundary conditions, by modifying the boundary trace space, is also an interesting direction.

\section*{Acknowledgments}

This work is partially supported by the National Science Foundation through grant DMS-2309802. QD is also partially supported by the National Science Foundation through grant DMS-2309245 and the US Department of Energy through grant DE-SC0025347. KR is also partially supported by the Gordon \& Betty Moore Foundation Award GBMF12801 (doi.org/10.37807/GBMF12801). 

{\small
\bibliography{BIB-REN}
\bibliographystyle{siam}
}

\appendix
 
\section{Proof of hypothesis \texorpdfstring{($\bH$)}{} for small data}
\label{SEC:Proof of Hypothesis  (H)}

Here we show that the hypothesis ($\bH$) can be verified in the case of small data.

Let us recall from Section~\ref{SUBSEC:Iterative Decomp} the notation
\begin{equation}\label{EQ:Psi}
   \psi(x,t):=q(t)-\sigma_*(x)\,t,
   \qquad\text{so that}\qquad
   g^{(k)}=f-\psi\bigl(\cdot,\wh u^{(k)}\bigr).
\end{equation}
Note then that $\psi(x,0)=0$, $\partial_t\psi=q'-\sigma_*$. Since $s>d/2$ makes $H^s(\Omega)$ a Banach algebra and $\psi(x,\cdot)$ is smooth and vanishes at $t=0$, a Moser-type estimate supplies a nondecreasing $\Lambda_\psi:[0,\infty)\to[0,\infty)$ with
$\|\psi(\cdot,v)\|_{H^s}\le \Lambda_\psi(\|v\|_{H^s})\|v\|_{H^s}$ (the subscript distinguishes it from the trace operator $\Lambda$ of Theorem~\ref{THM:H Partial}). Let
$C_{\wh S}:=\|\wh S_*\|_{\cH_{\mathrm{src}}\to\cH_{\mathrm{src}}}$, which is
finite because the decomposed surrogate has the form
$\wh S_*=T_{P_0}-\cR_Y\circ \bM \circ(\text{feature maps}){}-\wh D_*$: the parametrix
potential is bounded, $T_{P_0}:H^s(\Omega)\to H^{s+2}(\Omega)\hookrightarrow
H^s(\Omega)$ (transmission property~\cite{McLean-Book00}), and the learned
corrector is a bounded finite-rank map into the selected regularity-controlled output space, which embeds into $H^s(\Omega)$. The corrector itself need not have zero boundary trace. The defect-potential approximation $\wh D_*$ of Corollary~\ref{COR:Chord Rates} is likewise assumed to be a bounded finite-rank map into $H^s(\Omega)$. Note that $\wh S_*$ is \emph{not} itself finite-rank.

\begin{lemma}\label{LEM:H Small Data}
Fix $\rho_s>0$ with
\begin{equation}\label{EQ:Small Data Hs}
   C_{\wh S}\, \Lambda_\psi(\rho_s)\le\tfrac12
   \qquad\text{and}\qquad
   C_{\wh S}\,\|f\|_{H^s(\Omega)}\le\tfrac{\rho_s}{2},
\end{equation}
and define
\[
   R_*:=\|f\|_{H^s(\Omega)}+\Lambda_\psi(\rho_s)\,\rho_s,
   \qquad
   \cM_*:=\bigl\{g\in\cH_{\mathrm{src}}:\|g\|_{H^s(\Omega)}\le R_*\bigr\},
\]
with $\eta_*$ given by~\eqref{EQ:Eta} for this $\cM_*$. Assume $\wh u^{(0)}=0$ and
that for some $\rho>0$
\begin{equation}\label{EQ:Small Data Linf}
   \kappa_*(\rho):=\|\tau_{\sigma_*}\|_{L^\infty}\!\!\sup_{x\in\Omega,\,|t|\le\rho}\!\!
   \bigl|q'(t)-\sigma_*(x)\bigr|<1,
   \qquad
   \|\tau_{\sigma_*}\|_{L^\infty}\|f\|_{L^\infty(\Omega)}+\eta_*
   \le\bigl(1-\kappa_*(\rho)\bigr)\rho .
\end{equation}
Then for every $k$ one has $\|\wh u^{(k)}\|_{H^s}\le\rho_s$,
$\|\wh u^{(k)}\|_{L^\infty}\le\rho$, and $g^{(k)}\in\cM_*$. In particular, {\rm ($\bH$)} holds, and Theorem~\ref{THM:Chord Convergence} applies with the
constant $\kappa_*=\kappa_*(\rho)$.
\end{lemma}
\begin{proof}
Note first that $\rho_s$, $R_*$, $\cM_*$ and hence $\eta_*$ are defined from $f$,
$q$, $\sigma_*$ and $C_{\wh S}$ alone, so~\eqref{EQ:Small Data Linf} is well-defined. 
 
Put $a_k:=\|\wh u^{(k)}\|_{H^s}$, so $a_0=0$. If $a_k\le\rho_s$, then by~\eqref{EQ:Psi} and the Moser estimate
$\|g^{(k)}\|_{H^s}\le\|f\|_{H^s}+ \Lambda_\psi(a_k)a_k\le R_*$, so $g^{(k)}\in\cM_*$, and
since $\wh u^{(k+1)}=\wh S_*g^{(k)}$,
\[
   a_{k+1}\le C_{\wh S}\bigl(\|f\|_{H^s}+\Lambda_\psi(a_k)a_k\bigr)
   \le\tfrac{\rho_s}{2}+\tfrac{a_k}{2}\le\rho_s
\]
by~\eqref{EQ:Small Data Hs}. Induction gives $a_k\le\rho_s$ and $g^{(k)}\in\cM_*$
for all $k$, which is the membership required in {\rm ($\bH$)} and legitimizes the use of $\eta_*$ below.

Since $\psi(x,0)=0$, the mean value theorem gives $|\psi(x,t)|\le\Theta(\rho)|t|$ for $|t|\le\rho$, where
$\Theta(\rho):=\sup_{x,|t|\le\rho}|q'(t)-\sigma_*(x)|$. Assume inductively
$\|\wh u^{(k)}\|_{L^\infty}\le\rho$, true for $k=0$. Then
$\|g^{(k)}\|_{L^\infty}\le\|f\|_{L^\infty}+\Theta(\rho)\rho$. We apply~\eqref{EQ:Eta} together with the bound
$\|S_*g\|_{L^\infty}\le\|\tau_{\sigma_*}\|_{L^\infty}\|g\|_{L^\infty}$ from the proof
of Lemma~\ref{LEM:Chord Contraction} to get
\[
   \|\wh u^{(k+1)}\|_{L^\infty}
   \le\|S_*g^{(k)}\|_{L^\infty}+\eta_*
   \le\|\tau_{\sigma_*}\|_{L^\infty}\|f\|_{L^\infty}+\kappa_*(\rho)\,\rho+\eta_*
   \le\rho,
\]
the last step by~\eqref{EQ:Small Data Linf}. This closes the induction.
\end{proof}

Condition~\eqref{EQ:Small Data Linf} is a genuine smallness restriction on the data:
it fails once $\|f\|_{L^\infty}$ is large enough to push the solution into a range
where $q'$ deviates too much from $\sigma_*$.
 
Two constraints govern the choice of $\sigma_*$. Analytic regularity of $\sigma_*$ is required for the exponential output rate of Corollary~\ref{COR:Chord Rates}
(Section~\ref{SEC:Elliptic} assumes analytic coefficients), together with the uniform analytic data-space and sampling assumptions stated there. The Sobolev invariance result above does not by itself establish these analytic assumptions. The analytic regime excludes data-fitted choices such as $\sigma_*=q'(v)$ for a merely Sobolev-regular $v$: if $v\in H^{s+2}(\Omega)$, then $\sigma_*\in H^{s+2}(\Omega)$ by the Banach-algebra property, an $\cL_*$-harmonic function $w$ solves $\cL_0w=-\sigma_*w$ with right-hand side in $H^{s+2}(\Omega)$ at best, and elliptic regularity yields only $w\in H^{s+4}_{\mathrm{loc}}(\Omega)$. Hence, no output space satisfies the uniform analytic bounds of Corollary~\ref{COR:Global Rates}, and the output rate is capped at the algebraic order of the Sobolev regime. (The initial guess $\wh u^{(0)}=0$ of Lemma~\ref{LEM:H Small Data} gives $\sigma_*=q'(0)$, a constant, which is admissible. The exclusion concerns updated choices such as $\sigma_*=q'(\wh u^{(1)})$.) A \emph{constant} $\sigma_*\equiv c\ge0$ is the natural admissible choice. Subject to that, the oscillation factor $\sup_{t\in J}|c-q'(t)|$of~\eqref{EQ:Kappa} is minimized by the Chebyshev centre of the range of $q'$ on $J$~\cite{Cheney-Book82}. Since the companion factor $\|\tau_c\|_\infty$ is nonincreasing in $c$, this choice minimizes the upper bound obtained after replacing $\|\tau_{\sigma_*}\|_\infty$ by $\|\tau\|_\infty$, though not necessarily $\kappa_*$ itself. The center is
\[
   c=\tfrac12\Bigl(\min_{t\in J}q'(t)+\max_{t\in J}q'(t)\Bigr),
   \qquad\text{giving}\qquad
   \sup_{t\in J}|c-q'(t)|=\tfrac12\operatorname{osc}_{J}(q') .
\]
Picard ($c=0$) instead gives $\sup_{J}|q'|=\max_{J}q'$, so freezing at the Chebyshev
center at least halves the contraction factor and correspondingly widens the range of data for which $\kappa_*<1$. For the cubic $q(u)=u^3$ on $J=[-\rho,\rho]$ one has $q'=3t^2$, so $c=\tfrac32\rho^2$ and
$\sup_{J}|c-q'|=\tfrac32\rho^2$, whereas Picard gives $3\rho^2$; on the unit disk with
$\gamma\equiv1$, where $\|\tau\|_\infty=\tfrac14$, the conditions for convergence are
$\rho<\sqrt{8/3}\approx1.63$ and $\rho<2/\sqrt3\approx1.15$ respectively.

\end{document}